\documentclass{amsart}
\usepackage[latin1]{inputenc}
\usepackage{amsfonts}
\usepackage{amsmath}
\usepackage{graphicx}
\usepackage{latexsym}
\usepackage[all]{xy}
\usepackage{amsthm}
\usepackage{amssymb}

\date{}
\title{Semiflows induced by length metrics: On the way to extinction}

\author[Á.~Martínez-Pérez]{Álvaro Martínez-Pérez}
\address{Departamento de Análisis Económico y finanzas\\ Universidad de Castila-La Mancha\\ Avda Real Fábrica de Sedas s/n. Talavera de la Reina, 45600,  Spain\\ Telephone: +34 902 204 100. Fax +34  902 204 130.}
\email{alvaro.martinezperez@uclm.es}

\author[M.~A.~Morón]{Manuel A. Morón}
\address{Departamento de Geometría y Topología\\ Universidad Complutense de Madrid\\ Madrid, 28040,  Spain}
\email{mamoron@mat.ucm.es}
\thanks{The authors are partially supported by MTM-2009-07030.}

\begin{document}

\begin{abstract} Bing and Moise proved, independently, that any Peano continuum
admits a length metric $d$. We treat non-degenerate Peano continua with a length metric as evolution systems instead of stationary objects. For any compact length space $(X, d)$ we consider a
semiflow in the hyperspace $2^X$ of all non-empty closed sets in $X$. This semiflow starts with a canonical copy of the
Peano continuum $(X,d)$ at $t = 0$ and, at some time, collapses everything into a
point. We study some properties of this semiflow for several classes of spaces,
manifolds, graphs and finite polyhedra among them.
\end{abstract}

\maketitle

\newtheorem{definicion}{Definition}[section]
\newtheorem{nota}[definicion]{Remark}
\newtheorem{prop}[definicion]{Proposition}
\newtheorem{lema}[definicion]{Lemma}
\newtheorem{obs}[definicion]{Remark}
\newtheorem{teorema}[definicion]{Theorem}
\newtheorem{cor}[definicion]{Corollary}
\newtheorem{ejp}[definicion]{Example}
\newtheorem{contejp}[definicion]{Counterexample}
\newtheorem{quest}[definicion]{Question}

\tableofcontents

\begin{footnotesize}
\end{footnotesize}

\begin{footnotesize}
\end{footnotesize}

\section{Introduction}

Along this paper, $X$ will represent a non-degenerate  Peano continuum, that is  a connected and locally connected compact metrizable space with more than one point. As known, see the preliminaries, we can always define a geodesic metric in $X$ inducing the original topology.

Let $(Y,d)$ be a compact metric space. The hyperspace of all
non-empty closed sets in $Y$ is denoted by $2^Y$. The Hausdorff
metric, $d_H$, induced by $d$ on $2^Y$ is defined to be, given
$A_1,A_2 \in 2^Y$,
$$d_H(A_1,A_2):=max\{\sup_{x\in A_1}\{d(x,A_2)\},
\sup_{y\in A_2}\{d(y,A_1)\}\},$$ or equivalently,
$$d_H(A_1,A_2):= \inf\{\varepsilon>0 \ | \ A_1\subset
B(A_2,\varepsilon) \\ \mbox{ y } A_2\subset B(A_1,\varepsilon)\}.$$
$2^Y_H$ represents $2^Y$ endowed with the Hausdorff metric.

Our starting point comes from a result by S. B.  Nadler in \cite{N3}.
Considering a length metric on a Peano continuum $X$, we define a
semiflow in the hyperspace $2^X_H$ such that for every non-empty
compact subset $A$ and every positive $t$, the image is the
generalized closed ball about $A$ of radius $t$. This semiflow has
a global asymptotically stable attractor which is the point
$\{X\}\in 2^X$. This attractor is reached from any orbit at a
finite time which is at most equal to the diameter of $(X,d)$. In
this sense we say that this evolution system is {\it
extinguishable}.

Our main idea is to consider any non-degenerate Peano continuum $X$ with a length metric $d$ as an extinguishable evolution system
instead of an stationary object. By this way many natural questions appear. In this paper we get only few answers and some adequate examples.

In this semiflow any point  $x\in X$, considered as a unitary
closed subset of $X$, evolves following a geodesic in $2^X_H$ till
reaching the point $\{X\}\in2^X$. This geodesic, which is the
trajectory of $x$, takes values on closed metric balls centered at
$x$. In this sense the canonical copy $X$ inside $2^X_H$ evolves,
with constant speed $1$, through the subspace of real closed
metric balls till the extinction which is produced just at the
time equal to the diameter of $(X,d)$. Then from time $t=0$ to
$t=diam(X)$ we pass from an isometric copy of $(X,d)$ to a trivial
space (i.e. a single point) through geodesics.  Consequently, many
questions naturally arise. For example, how many different
topological types (homotopy types) of spaces appear throughout the
evolution till the extinction? Since any point in the canonical
copy is moving along a geodesic, a general principle is in order:

{\it The whole canonical copy is moving under a minimum energy
principle till the extinction.}

What does this principle imply regarding the topological types
(homotopy types, etc.) you have to reach in the evolution?
Specifically, for each time $t \in [0,diam(X)]$ the semiflow
converts the canonical copy $X$ into the space $X_{t}$ which is
the subspace of all closed balls with radius $t$ with the
Hausdorff metric. How are topologically (homotopically, etc.)
related $X$ to $X_{t}$? Note that $X_{0}=X$ and $X_{diam(X)}$ is a
point. What kind of properties of $(X,d)$ are {\it positively
invariant} under the action of the semiflow? A property P is said
to be positively invariant when if $X$ has P, this implies $X_{t}$
has P for all $t>0$. Obviously, the property of being a Peano
continuum is positively invariant.

Another interesting question is to detect what kind of properties
are {\it Weierstrass-type} for the evolution induced by the
semiflow. Without treating to give an exhaustive definition of
this terminology,  we recall that the classical Weierstrass
Theorem in Real Analysis, later extended to Topology, asserts that
any continuous real function on a closed finite real interval is
bounded and it attains its maximal and minimum value. Think about
the semiflow as a {\it continuous} function $t\longrightarrow
X_{t}$ defined on the real interval $[0,diam(X)]$. Any consistent
definition of a property P to be {\it bounded} for the evolution
should imply that if the semiflow reach only a finite number of
different P-type then P has to be bounded.

Associated to the above paragraphs, there is also a problem on the
stability of the topological type of the canonical copy $(X,d)$.
This problem is related to the existence of a positive time
$\varepsilon_{0}$ such that the semiflow behaves like a flow in
the subspace of closed balls for $t\in [0,\varepsilon_{0})$. This
is a kind of parallelizability   property of the semiflow near
$t=0$. It is equivalent to a very natural geometric question: If
$(X,d)$ is a Peano continuum with a geodesic metric, is it true
that there is a positive real $\varepsilon_{0}$ such that each
closed ball of radius lower than $\varepsilon_{0}$ determines
univocally its center? We call this property as the {\it
topological robustness} of $(X,d)$. It is obviously invariant by isometries.
We show that this property
strongly depends on the (metric) geometry of $(X,d)$ and not on
the topology of $X$. In fact we give an example of bilipschitz
homeomorphic pairs,
 one of them being topologically robust and the other not.

After some preliminaries we describe, in Section $3$, the semiflow
and some of its basic properties. The main element is what we call
the {\it dynamical cone} of $(X,d)$ which is just a dynamical view
of the subspace of closed balls with the Hausdorff metric. Using
this we give a homological model of the evolution. We introduce a
Lyapunov function for the semiflow which plays the role of a
potential function on the canonical copy. It allows us to define
{\it centers}, as points of minimal energy, and {\it extremes}, as
points of maximal energy. Some examples are
 provided.

In section $4$ we focus on topological robustness giving some
positive and negative results. We also give an example of a
compact geodesic space that has to pass through countable many
different topological type before the extinction, although the
homotopy type remains constant.

In section $5$, we prove that being, topologically, a finite graph
is a positively invariant property for the semiflow as it is also
being a finite tree. We study the semiflow for finite graphs with
natural geodesic metrics in some depth. We prove that, in this
case, the topological properties of
 the levels for the semiflow are all {\it bounded} in the sense that the canonical copy only goes through a finite number of topological types until
 the extinction. This means that in the framework of finite metric graphs all the topological properties are {\it Weierstrass-type} properties. This gives
 some meaning to our minimum energy principle. We also put examples that prevent on the monotonicity on the changes of
 topological types in the sense that our example is a graph with the homotopy type of a $1$-sphere and on the way to extinction it has to pass through the
 homotopy type of the figure eight.

In Section $6$ we point out that the Whitney functions in hyperspaces are intrinsically related to Lyapunov functions for the semiflow. Using this we give some geometric model for the semiflow in terms of geodesically complete $\mathbb{R}$-trees and their corresponding end spaces and give a characterization of the topological robustness.

The work in this paper can be extended to the non-compact case by considering complete connected proper metric spaces and the hyperspace of compact subsets
with the Hausdorff metric.

\section{Preliminaries}

To avoid introducing too many concepts which are unnecessary for this work,
we are going to introduce length spaces in a very restricted way. Thus, instead of talking about length structures, in which we must fix a set of admissible paths and a measure for them, see \cite{Bu-Bu},
we are going to start with a metric space and consider the length structure induced by the metric when all the paths are admissible. For the general framework of length spaces we also follow \cite{B-H}.

\begin{definicion} Let $(X,d)$ be a metric space. The length
$l(c)$ of a path $c:[a,b] \to X$ is \[l(c)= \underset{a=t_0\leq
t_1 \leq \cdots \leq t_n=b}{sup}
\sum_{i=0}^{n-1}d(c(t_i),c(t_{i+1})),\] where the supremum is taken among all possible partitions of the interval $a=t_0\leq t_1
\leq \cdots \leq t_n=b$.
\end{definicion}

\begin{definicion} Let $(X,d)$ be a metric space. $d$ is a \emph{length metric} if the distance between every pair of points $x,y\in X$ is equal to the infimum of the lengths of the paths joining them. (If there is not such a path then $d(x,y)=\infty$). If $d$ is a length metric, then $(X,d)$ is called \emph{length space}.
\end{definicion}

\begin{definicion}\label{camin_geod} Let $(X,d)$ be a metric space. A \emph{geodesic path} from $x\in X$ to $y\in X$ is a map $c$ from a closed interval $[0,l]\subset \mathbb{R}$ to $X$ such that $c(0)=x$, $c(l)=y$ and
$d(c(t),c(t'))=|t-t'| \ \forall t,t'\in [0,l]$. In particular,
$l=d(x,y)$. The image of $c$ is called \emph{geodesic segment} with endpoints $x$ and $y$ and it is denoted as $[x,y]$. When the context is clear, we may abuse of the notation and refer to the geodesic segment simply as geodesic.
\end{definicion}

\begin{definicion} If $X$ is a metric space such that for every pair of points there is a geodesic path joining them, then $X$ is said to be
\emph{geodesic}.
\end{definicion}

In general, not every length space is geodesic. Let us consider, for example, the euclidean plane without the origin. In this case, there is no path from $x$ to its symmetric with respect to the origin, $-x$, with length $d(x,-x)$, although it is clear that there are paths between them whose length is as close to that distance as we want.

Next result was proved, independently, by R.
H. Bing and E. Moise, in \cite{Bi} and \cite{Moi} respectively, in 1949.

\begin{teorema}[Bing and Moise] Every Peano continuum $(X,\tau)$
admits a metric $d$ such that $(X,\tau)$ and $(X,d)$ are homeomorphic and $(X,d)$ is a length space.
\end{teorema}

The following is a slightly weak version for length spaces of the Hopf-Rinow theorem (see \cite{Hopf}).
It is also known as Hopf-Rinow-Cohn-Vossen Theorem, see \cite[Theorem 2.5.28]{Bu-Bu}.

\begin{prop}\cite[Proposition 3.7]{B-H}
Let $X$  be a length space. If $X$ is complete and locally compact, then:
\begin{itemize}\item[(1)] Every closed bounded set of $X$ is compact; \item[(2)] $X$ is a geodesic space.
\end{itemize}
\end{prop}

\begin{obs} In this work we are dealing with geodesic compact spaces. Notice that, using the results above, we may consider on any Peano continuum a metric for which the space is geodesic.
\end{obs}

Notice that the existence of a geodesic path doesn't mean that it should be unique.

\begin{ejp} Consider the  graph from Figure \ref{ejp_1}
with the natural metric where every edge has length 1.
\end{ejp}

\begin{figure}[ht]
\includegraphics[scale=0.3]{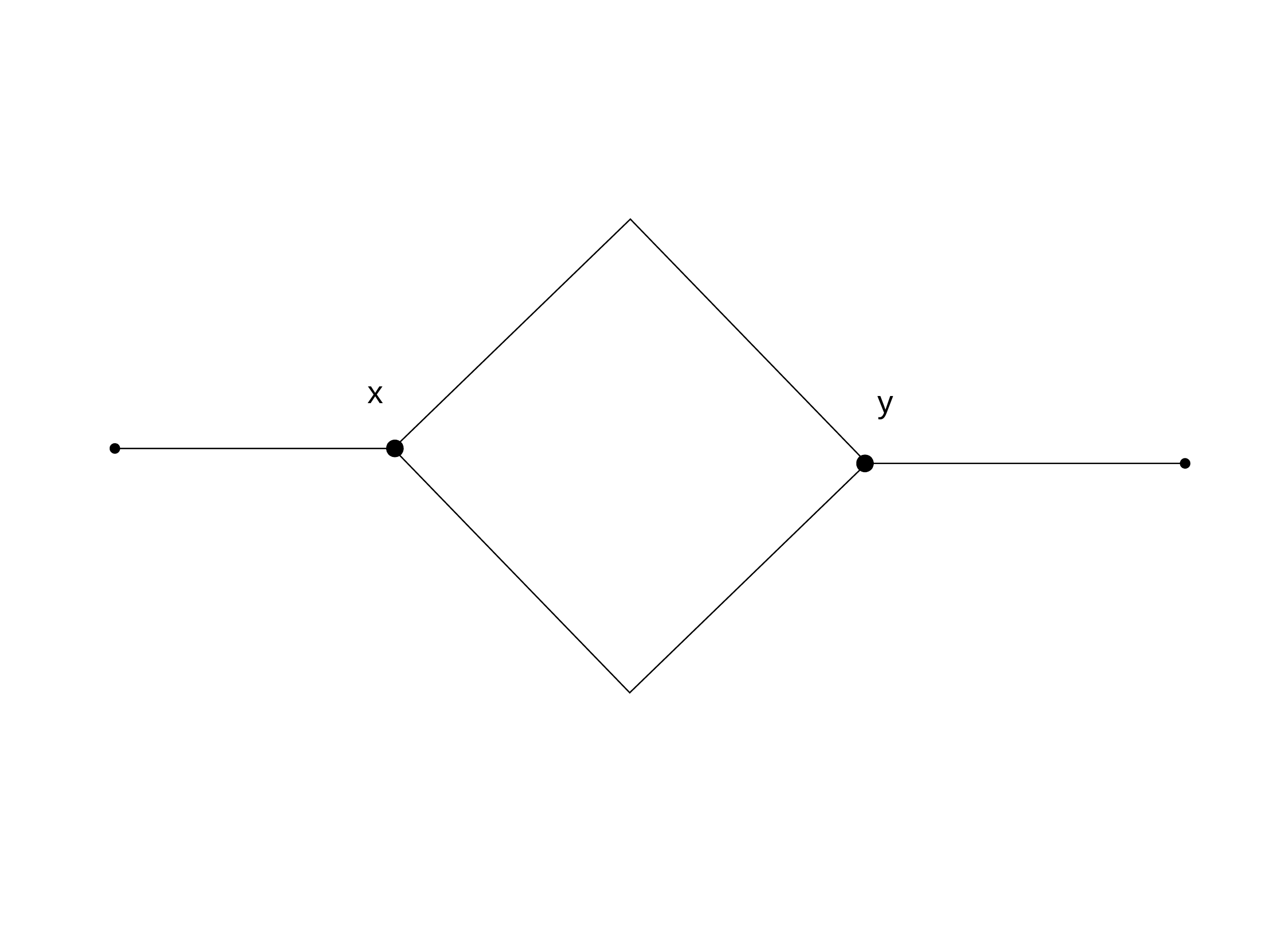}
\caption{The geodesic path joining two points need not be unique.}\label{ejp_1}
\end{figure}

As we can see, there are two geodesic paths (of length 2) joining $x$ to $y$.

For the next basic definitions and properties we follow the book from Bhatia and
Szegö, \cite{B-S}.

\begin{definicion} A \emph{dynamical system} on $X$ is the triplet $(X,\mathbb{R},\pi)$
where $\pi$ is a map from the product space $X\times \mathbb{R}$ into $X$ satisfying the following axioms:
\begin{itemize}\item[(i)] $\pi(x,0)=x \quad \forall x\in X$
\item[(ii)] $\pi(\pi(x,t),s)=\pi(x,t+s)$ for every $x\in X$ and
$t,s\in \mathbb{R}$. \item[(iii)] $\pi$ is continuous.
\end{itemize}
\end{definicion}

For any point $x\in X$, $\pi(\pi(x,t),-t)=x$. From this, it is readily seen the following result.

\begin{prop}\label{tray_hom} For every $t\in \mathbb{R}$, the map $\pi_t:X \to X$
defined by $\pi_t(x)=\pi(x,t)$ is a homeomorphism of $X$ in itself.\hfill$\square$
\end{prop}

Let $\Lambda^+(x):=\{y\in X \ | \ \mbox{ there is a sequence }
\{t_n\} \mbox{ in } \mathbb{R} \mbox{ with } t_n\to +\infty \mbox{
and } \pi(x,t_n) \to y\}$.

$A(M):=\{x\in X \ | \ \Lambda^+(x)\neq \emptyset \mbox{ and }
\Lambda^+(x)\cap M \neq \emptyset\}$.

\begin{definicion} A set $M$ is \emph{positively invariant}
if for every $x\in M$ and every $t>0$, $\pi(x,t)\in M$.
\end{definicion}

\begin{definicion} A set $M$ is said to be \emph{stable}
if every neighbourhood $U$ of $M$ contains a positively invariant neighbourhood
$V$ of $M$.
\end{definicion}

\begin{definicion} A set $M$ is an \emph{attractor} if $A(M)$
is a neighbougood of $M$.
\end{definicion}

\begin{definicion} A set $M$ is said to be \emph{asymptotically stable} if it is stable and an attractor.
\end{definicion}

\paragraph{\textbf{Lyapunov functions}}

\begin{teorema}\cite[Chapter V, Theorem 2.2]{B-S} A compact $M\subset X$ is asymptotically stable if and only if there exist a real valued continuous function, $\Phi$, defined in a neighbourhood $N$ of $M$ such that:

\begin{itemize}\item[(i)] $\Phi(x)=0$ if $x \in M$ and $\Phi(x)>0$ if $x \not \in M$;
\item[(ii)] $\Phi(x,t)<\Phi(x)$ for $x\not \in M$, $t>0$ and
$x[0,t]\subset N$.
\end{itemize}
\end{teorema}

\begin{definicion} This function $\Phi$ is a \emph{Lyapunov function} on $N$ for $\pi$.
\end{definicion}

\paragraph{\textbf{Semidynamical systems}}$\\$

Let us denote by $\mathbb{R}^+$ the interval $[0,\infty)$ in the real line.

\begin{definicion} A semidynamical system on $X$ is a triplet $(X,\mathbb{R}^+,\pi)$
where $\pi$ is a map from the product space $X\times \mathbb{R}^+$ into $X$
satisfying the following axioms:
\begin{itemize}\item[(i)] $\pi(x,0)=x \quad \forall x\in X$
\item[(ii)] $\pi(\pi(x,t),s)=\pi(x,t+s)$ for every $x\in X$ and
$t,s\in \mathbb{R}^+$. \item[(iii)] $\pi$ is continuous.
\end{itemize}
\end{definicion}

\begin{nota} Since the action of $\mathbb{R}^+$ is not reversible as in the case of dynamical systems, the behavior of the trajectories, $\{\pi_x(t):t\in \mathbb{R}^+\}$, in the semiflow is substantially different. Thus, contrary to the case shown in \ref{tray_hom}, the map $\pi_t:X
\to X$ defined by $\pi_t(x)=\pi(x,t)$ need not be a homeomorphism. Nevertheless, the definitions above referred to properties when $t\to +\infty$ as stable, asymptotically stable or attractor, work as well for semidynamical systems.
\end{nota}

\begin{ejp} Consider $X=[0,1]$ and $\pi(x,t)=\min\{x+t,1\}$.
Clearly, this is a semidynamical system but for $t\geq 1$,
$\pi(x,t)=1 \ \forall x\in X$. In fact, 1 is an asymptotically stable set for this semiflow.
\end{ejp}

\section{Definition and basic properties of the semiflow}

The basic conceps used below can be found in \cite{N2}.

Let $(X,d)$ be a compact length space. Since $X$ is compact, the closed subsets are compact and $2^X=\{A\subset X \ |
\ A \mbox{ nonempty }$ $\mbox{and compact}\}$. As we mentioned in the introduction,
$2^X_H$ represents $2^X$ with the Hausdorff metric.

Consider the map $\pi:2_H^X \times [0,\infty)
\to 2_H^X$ such that for any compact set $A$ and any $t\geq0$,
$\pi(A,t):=B^c(A,t)=\{x\in X : d(x,A)\leq t \}$, this is, the
generalized closed ball in $X$ about $A$ of radius $t$. We understand that $B^c(A,0)=A$.
\begin{prop}
The triplet $(2_H^X,\mathbb{R_+},\pi)$ defines a semidynamical system.\hfill$\square$
\end{prop}

The proof  can be found in \cite{N3}
although with  a different language.

Note that when $X$ is a length space any generalized closed ball
is the closure of the generalized open ball,
$B^c(A,\varepsilon)=\bar{B}(A,\varepsilon)$. Let us refer to $\partial
\bar{B}(A,\varepsilon)=\partial B(A,\varepsilon)=S(A,\varepsilon)=\{z\in X
\ | \ d(z,A)=\varepsilon \}$ as the \emph{border} of the ball and its points as \emph{border points}.
From now on, we will denote the closed
ball as $\bar{B}(A,\varepsilon)$.

\begin{obs}\label{atractor} Since $X$ is a compact metric space, for any $A\subset
X$ there exists some $t_A$ such that for every $t\geq t_A$,
$\bar{B}(A,t)=X$ and therefore, $\pi(A,t)=\{X\}\in 2^X$.
\end{obs}

Let us state a few basic properties about the hyperspace in relation to this map.

\begin{lema}\label{increasing} Let $X$ be a compact connected length space, $A\subset X$ and $0<\varepsilon_0<\varepsilon_1$
such that $\bar{B}(A,\varepsilon_0)\neq X$. Then
$\bar{B}(A,\varepsilon_0)\subsetneq \bar{B}(A,\varepsilon_1)$.
\end{lema}

\begin{proof} Otherwise, let us suppose that $\bar{B}(A,\varepsilon_0) =
\bar{B}(A,\varepsilon_1)$. Let
$\delta=\frac{1}{2}(\varepsilon_0+\varepsilon_1)$,
$\bar{B}(A,\varepsilon_0)\subset B(A,\delta)\subset
\bar{B}(A,\delta)\subset \bar{B}(A,\varepsilon_1)$ and all those balls coincide. Therefore, there is a proper subspace which is open and closed. This contradicts the fact that $X$ is connected.
\end{proof}

\begin{prop} For every $\varepsilon < diam(X), \ P(\varepsilon,X):=\{A\subset X \ | \ A
\mbox{ closed and } \\ diam(A)\geq \varepsilon\}$ is a neighborhood
of $\{X\}$ in $2^X_H$.
\end{prop}

\begin{proof} Let $\delta < \frac{diam(X)-\varepsilon}{2}$.
Consider any closed subset $A\subset X$ such that
$d_H(A,X)<\delta$. Since
$X$ is compact there are two points $x,y\in X$ such that
$d(x,y)=diam(X)$. Then,
$d(x,A),d(y,A)\leq \delta$ implies that $diam(A)\geq diam(X)-2\delta>\varepsilon$ and $A\in P(\varepsilon,X)$.
\end{proof}

\begin{prop} $\forall \varepsilon < diam(X)$
there exists a strong deformation retraction from $2^X_H$ onto
$P(\varepsilon,X)$.
\end{prop}

\begin{proof} For every closed subset $A\subset X$ there exists some $t_A\geq 0$ such that $t_A:=\inf\{t \ | \
diam(\bar{B}(A,t))\geq \varepsilon\}$. If $diam(A)\geq
\varepsilon$ consider $t_A=0$. The assignment $A\mapsto t_A$ defines a continuous real function on $2^X_H$ because of the continuity of the diameter function. Let us define the homotopy $G: 2^X_H \times
I \to 2^X_H$ as follows: $G(A,t)= \bar{B}(A,\min\{t\cdot
diam(X),t_A\})=\pi(A,\min\{t\cdot
diam(X),t_A\})$. $G$ is continuous, $G_0$ is the identity,
$G_1(2^X_H)\subset P(\varepsilon,X)$ and $G_t|_{P(\varepsilon,X)}$
is the identity.
\end{proof}

Suppose always that $(X,d)$ is a non-degenerate Peano continuum with a geodesic metric. It is clear that the points in the trajectories of the semiflow $\pi$,  going from single points
$\{x\}\in 2^X$ to the whole space $\{X\}\in 2^X$, are always closed balls centered at points. Hence, if we restrict $\pi$ to the subspace $\mathcal{B}\subset 2^X$ of all closed balls centered at points of $X$ we still have a semiflow.

\begin{prop} The subspace $\mathcal{B}$ of $2^X$ (or $C(X)$) is positively invariant for the semiflow $\pi$ and it is
 closed in $2^X_H$ (or $C(X)$) with the Hausdorff metric. Moreover, it is contractible and the canonical copy, considered as the subset of closed balls of radius zero, is a Z-set inside $\mathcal{B}$  in the sense that the identity in $\mathcal{B}$  is uniformly approximated by maps missing $X$.
\end{prop}

\begin{proof} Given a  ball $\bar{B}(x,\varepsilon)$ an a non-negative  real number $t$, we have
\[\pi(\bar{B}(x,\varepsilon), t)= \pi(x,\varepsilon + t)=\bar{B}(x,\varepsilon+t).\] Since $\mathcal{B}=\pi(X\times [0, diam(X)])$, then it is compact.
Moreover $\pi:\mathcal{B}\times [0, diam(X)]\longrightarrow \mathcal{B}$ defines an strong deformation retraction from $\mathcal{B}$ to the point
$\{X\}\in 2^X$ representing the whole space. The sequence of maps $\pi_{n}:\mathcal{B}\longrightarrow \mathcal{B}$, defined by $\pi_{n}(B)=\pi(B, \frac{1}{n})$ converges uniformly to the identity. Moreover $\pi_{n}(\mathcal{B})\bigcap X=\emptyset$.
\end{proof}

Thus, the semiflow in $\mathcal{B}$ takes the isometric copy of the original metric space (given by the single points with the Hausdorff metric in the Hyperspace) to the point $\{X\}$ as we saw in \ref{atractor}.

For obvious reasons we call $\mathcal{B}$ the {\it dynamical cone} of $(X,d)$.

Further questions the space $\mathcal{B}$ naturally arise. For example, its topological dimension, local properties, when is it an absolute retract?, etc. We are not going to follow this line herein. We  will probably do it in a future work.

\paragraph{\bf{Homological model for the semiflow}}$\\$

We use the book of Hatcher, \cite{Hat}, for undefined concepts and
notations related to homology. Let us denote by $X_{0}$ the
canonical copy of $X$ inside $2^X_{H}$ and by $[X_{0},X_{t}]$ the
set $\pi(X_{0}\times[0,t])$. Note that
\[[X_{0},X_{t}]=\{A\in 2^X \, | \, \exists  \ \varepsilon \in [0,t],  x\in X_{0}   \text{ with}   \   A= \bar{B}(x,\varepsilon)\}.\]
The following is clear
\begin{prop}
The semiflow $\pi$ induces a strong deformation retraction from $[X_{0},X_{t}]$ onto $X_{t}$. Concretely
\[G:[X_{0},X_{t}]\times [0,1]\longrightarrow [X_{0},X_{t}] \ \text{defined by}\ G(\bar{B}(x,\varepsilon), s)=\bar{B}(x,(1-s)\varepsilon +st)\] is a strong deformation retraction onto $X_{t}$.
\end{prop}
 So, the singular relative homology groups $H_{n}([X_{0},X_{t}],X_{0})$ measure the difference between the homology of the semiflow at time $t$, because $H_{n}([X_{0},X_{t}])$ is isomorphic to $H_{n}(X_{t})$,  and  the homology of the initial condition $X_{0}$.

 Given a compact geodesic space $(X,d)$ and for every natural number $n$ we can define a transformation
 \[H_{n}:[0, diam(X)]\longrightarrow AbGroups \ \text{defined by} \ H_{n}(t)=H_{n}([X_{0},X_{t}],X_{0}) \] which transform  non-negative real numbers into abelian groups. The long exact sequence for relative singular homology and the fact that $[X_{0},X_{diam(X)}]= \mathcal{B}$ is contractible,   allow us to calculate the values at the extremes of the interval.
 \begin{prop}
 For any  compact geodesic metric space $(X,d)$ and for every natural number $n$ we have:
 $H_{n}(0)$ is the trivial group and  $H_{n}(diam(X))\equiv \widetilde{H}_{n-1}(X).$
 \end{prop}
 Once we  have  the definition of the dynamical cone of a compact
 geodesic metric   space $(X,d)$,  we can define the {\it dynamical suspension} of
 $(X,d)$ as the space obtained from the dynamical cone
 $\mathcal{B}$ by collapsing the canonical copy $X_{0}$ of $X$ in
 $\mathcal{B}$ to a point. The result bellow detects a similarity
 of behavior between the dynamical and the usual suspension. The
 proof relies on the fact that when $(X,d)$ is a topologically robust compact metric
 geodesic space (see Definition \ref{Def: robust}) then  $(\mathcal{B}, X_{0})$ is a good pair in
 the sense of \cite{Hat}.
 \begin{prop}
 Suppose that $(X,d)$ is a topologically robust compact metric
 geodesic space and denote by $\mathcal{S}$ the dynamical
 suspension of $(X,d)$. Then $\widetilde{H}_{n}(\mathcal{S})
 \equiv \widetilde{H}_{n-1}(X).$
 \end{prop}

\paragraph{\bf{Order arcs}}$\\$

Next, let us focus on the trajectories, $\{\pi_A(t):t\in \mathbb{R}^+\}$, for any $A\in 2^X$.

We will see that from every point in the hyperspace, the trajectory on the semiflow is what is called an order arc from the initial point $A$ to $\{X\}$.

Also, this trajectories are geodesic paths with the Hausdorff metric. This means that we can see the semiflow in the hyperspace as a minimal energy flow in which each point is sent to the global attractor through a minimal path.

We extract the following definitions from \cite{I-N}

\begin{definicion} A collection $\mathcal{N}$ of sets is a
\emph{nest} provided that for any $N_1,N_2\in \mathcal{N}$,
$N_1\subset N_2$ or $N_2\subset N_1$.
\end{definicion}

\begin{definicion} Let $X$ be a compactum, and let $\mathcal{H}\subset
2^X$. An \emph{order arc in $\mathcal{H}$} is an arc, $\alpha$, in
$\mathcal{H}$ such that $\alpha$ is a nest.
\end{definicion}

\begin{definicion} A \emph{nest from $A_0$ to $A_1$} is a nest,
$\mathcal{N}$ such that $A_0,A_1\in \mathcal{N}$ and $A_0\subset N
\subset A_1$ for all $N\in \mathcal{N}$.
\end{definicion}

For any $A \in X$, with $X$ a compact length space (and therefore
connected), $N_A=\{\bar{B}(x,\varepsilon)\, | \, \varepsilon\geq 0 \}$ is an
order arc and, in particular, a nest from $\{A\}$ to $\{X\}$.

\begin{prop} The trajectories of the semiflow are geodesic paths
in the hyperspace with the Hausdorff metric.
\end{prop}

\begin{proof} In general, $\forall A \in \mathcal{H}$,
$N_A=\{\bar{B}(A,\varepsilon) \, | \, \varepsilon\geq 0 \}$, the generalized
closed balls, defines an order arc. Let
$t_A:=inf_{t\geq0}\{\bar{B}(A,t)=X\}$. Then the path
$\alpha_A:[0,t_A]\to \mathcal{H}$ with $\alpha_A(t)=\bar{B}(A,t)$
is a geodesic path from $A$ to $X$ where $d_H(\alpha_A(t),X)=t_A-t
\ \forall A\in \mathcal{H}$ and $\forall \, 0\leq t\leq t_A$. Thus,
the trajectories in the semiflow are geodesic paths in the
hyperspace with the Hausdorff metric.
\end{proof}
\paragraph{\bf{Lyapunov functions}}

Since we have a semiflow with a global asymptotically  stable attractor, it is natural to define a Lyapunov function for  it.

Clearly, $\{X\}\in \mathcal{H}$ is an asymptotically stable equilibrium point  for the semiflow.

Following \cite{B-S}, the first step to define a Lyapunov function is to define the map $\Phi(A)=d_H(A,X)$ and then, with some technical work, make it decreasing in the orbits. In this case, it is trivial to check that $\Phi$ is already a Lyapunov function and we can avoid the rest of the construction
which does not provide any benefits.

Consider $\Phi$, as above, such that $\Phi(A)=d_H(A,X)$. This function, restricted to the isometric copy of $X$ in $2^X_H$ allows us to define a function that, in some sense, plays the role of a potential on the hyperspace:
\[\Phi|_X:X \to [m,M] \mbox{ where } \Phi(x)=d_H(x,X)=max\{d(x,y) \ |
\ y\in X)\}.\]

Note that $M=diam(X)$ and $m\geq M/2$.

This function yields a decomposition of the space in equipotential subspaces $\{\Phi^{-1}(t) \, | \ t\in [m,M] \}$.

\begin{obs} For any isometry $f:X\to X$ and any
$x\in X, \ \Phi(x)=\Phi(f(x))$.
\end{obs}

Let us denote as \emph{centers} the points where this function takes the minimum value, $\Phi^{-1}(m)$, and as \emph{extrema} the points where it takes the maximum value, $\Phi^{-1}(M)$.

\begin{ejp} \begin{picture}(1,1)(-50,40)
    \put(0,0){\line(1,0){100}}
    \put(50,0){\circle*{4}}
    \put(0,0){\circle*{4}}
    \put(100,0){\circle*{4}}
    \put(45,10){$\Phi^{-1}(m)$}
\end{picture}
\end{ejp}

\vspace{2cm}

Here, the center is the middle of the segment while the extrema are the end points of the interval.

\begin{ejp} Consider the circle: $X=\{e^{2\pi i \cdot x} \ | \ 0\leq x <1\}\subset \mathbb{C}$
with the length metric.
\end{ejp}

Clearlly, $\Phi(x)=\pi \ \forall x\in X$.

\begin{ejp} Let $Q=[0,\frac{1}{n}]^\mathbb{N}$, the  Hilbert cube, with the $l_2$ metric.
\end{ejp}

Then, \[M= \sum_{n\in \mathbb{N}} \frac{1}{n^2} \ \mbox{ and } \
\Phi^{-1} (M)=\{(x_n) \ | \ x_n=\{0,\frac{1}{n}\} \ \forall n\}.\]
while \[m= \sum_{n\in \mathbb{N}} \frac{1}{4n^2} \ \mbox{ and
} \ \Phi^{-1} (m)=\Big{(}\frac{1}{2n}\Big{)} .\]

\section{Topological robustness of length spaces: examples and counterexamples}

Let us define $p_\varepsilon: X \to 2^X_H$ such that
$p_\varepsilon(x):=\bar{B}(x,\varepsilon) \ \forall x\in X $. We already know that at level 0 we have an isometric copy of the space and that there is a level, $\varepsilon_0$, such that $p_\varepsilon(X)$ is a single point $\forall \, \varepsilon \geq \varepsilon_0$.
The problem is to understand how are these projections $p_\varepsilon(X)$. The difficulty to give general results comes from the fact that, even for easy examples, the projection might be more complex (topologically or even homotopically) than the original metric space.

\begin{ejp} \label{horm} Consider the graph in Figure \ref{fig.1} with the geodesic metric and every edge of length 1.
\end{ejp}

\begin{figure}[ht]
\includegraphics[scale=0.5]{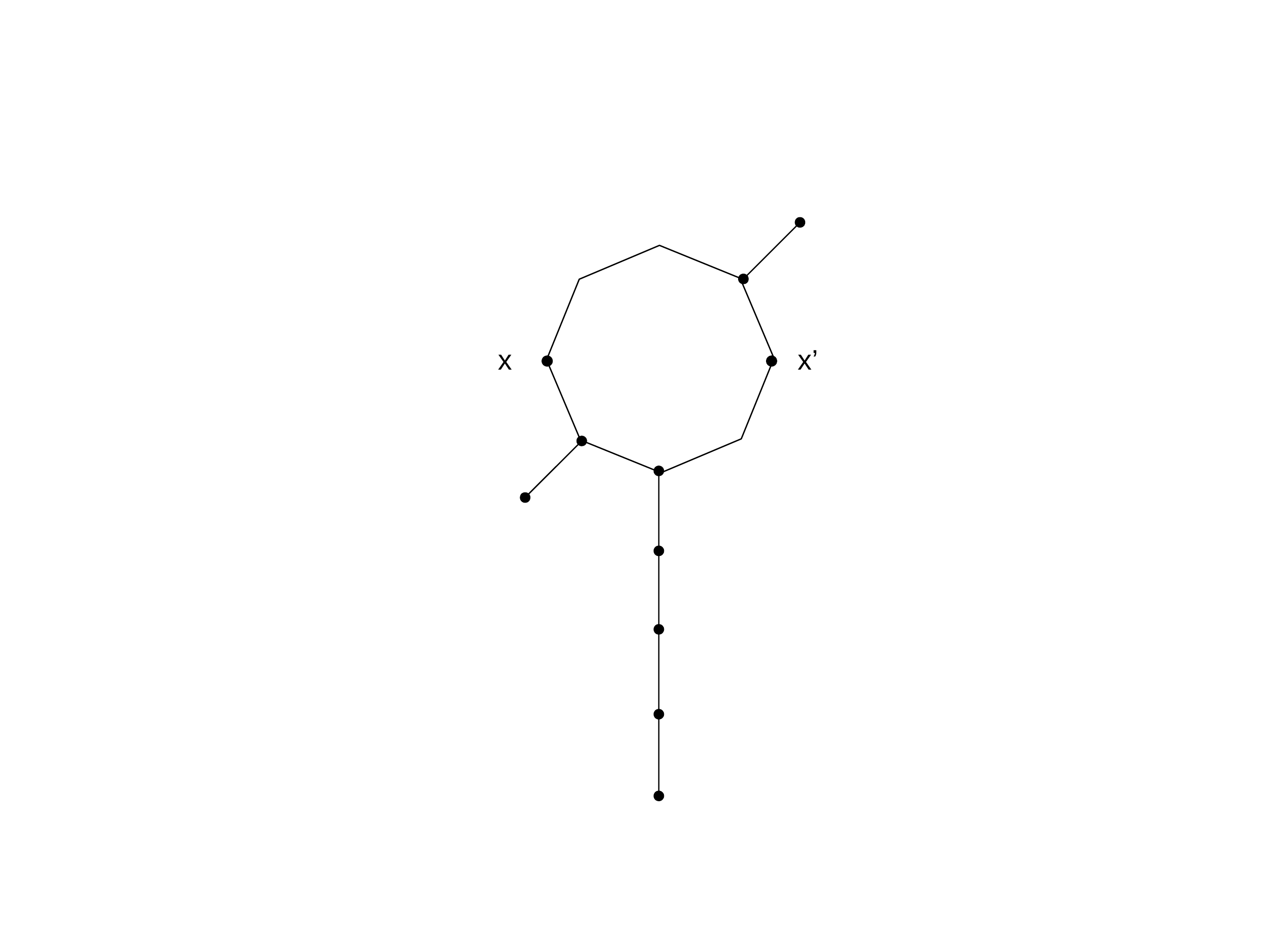}
\caption{The projection need not be homotopically dominated by the original space.}\label{fig.1}
\end{figure}

Considering $\varepsilon=4$, the balls about $x$ and $x'$
coincide and the projection is something homeomorphic to Figure \ref{fig.2}:

\begin{figure}[ht]
\includegraphics[scale=0.5]{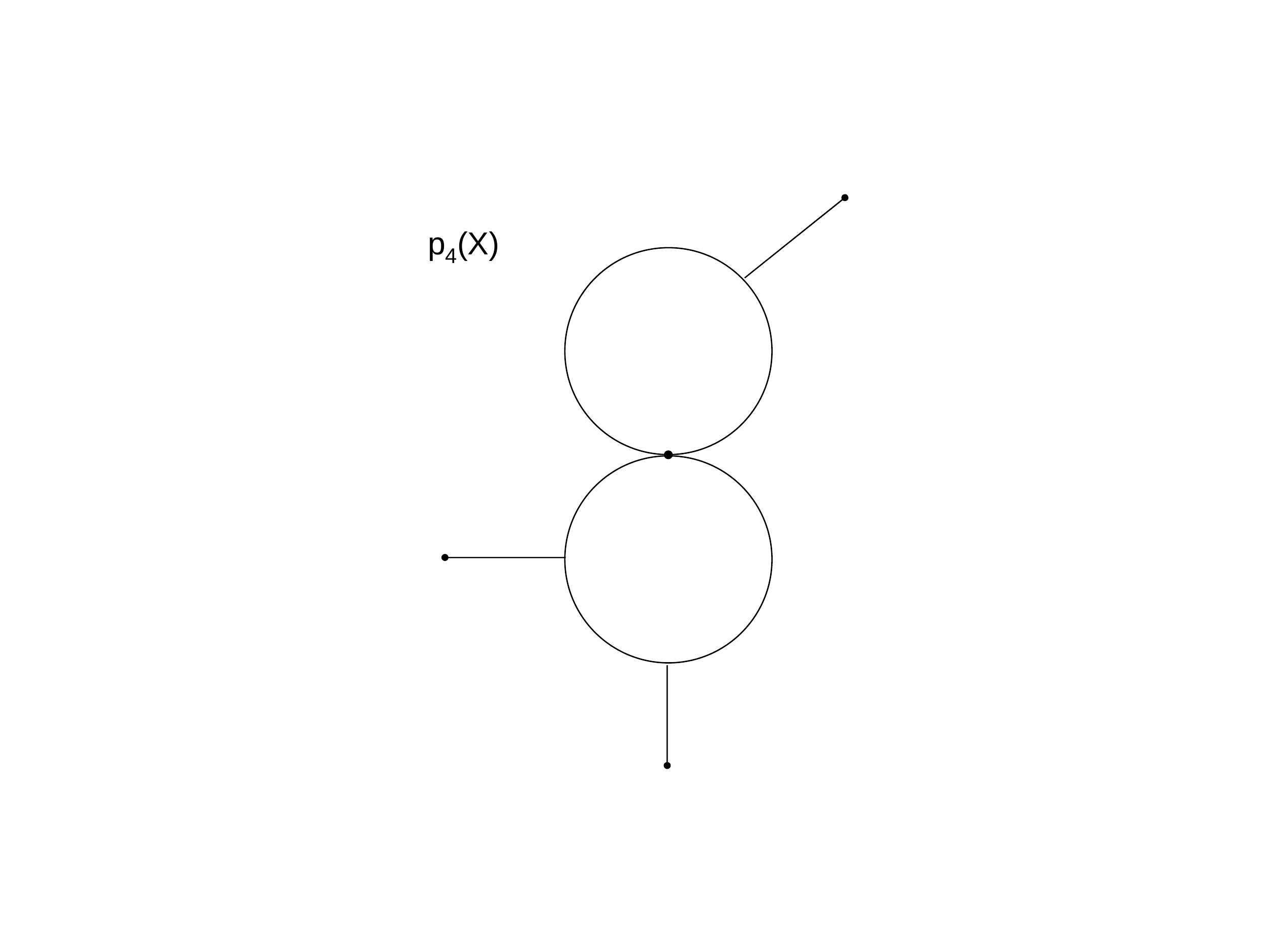}
\caption{Projection of the graph from Figure \ref{fig.1} for $\varepsilon=4$.}\label{fig.2}
\end{figure}

\begin{definicion}\label{Def: robust} A compact length space $(X,d)$ is \emph{topologically robust} if there is some $\varepsilon_0$ such that
$\forall \varepsilon\leq \varepsilon_0$, $p_\varepsilon$ is a topological embedding.
\end{definicion}

Our first aim is finding conditions on
$(X,d)$ to assure that it is topologically robust. As we saw in the introduction, this means that the semiflow $\pi$ keeps the topological type of $X$ for some time $\varepsilon_0$. It is clear that if for some $\varepsilon'$ $p_{\varepsilon'}: X \to 2^X_H$ is a topological embedding then the same holds for $0\leq\varepsilon\leq \varepsilon'$.

Hyperspaces of Peano continua are Hilbert cubes (see \cite{C-S}). This implies that there exists an embedding of  $X$ as $Z$-set in the Hilbert cube in such a way that there are homeomorphic copies of $X$ in its complement and as close to $X$ as we want if we consider the
hyperspace $2^{2^X}$ with the Hausdorff metric $d_{H^2}$.

\begin{obs} $d_{H^2}(X,p_\varepsilon (X))=\varepsilon$.
\end{obs}

Not every compact connected length space is topologically robust.

\begin{contejp} Consider in the real plane the space: $$X=\{(x,0)\, | \, 0\leq x\leq 1\}
\cup \Big{\{} \underset{n\in \mathbb{N}}{\cup}
\{(\frac{1}{2^n},y) \, | \, 0\leq y\leq \frac{1}{2^n}\}\Big{\}}$$ with the
natural length metric. See Figure \ref{fig.3}.
 \end{contejp}

\begin{figure}[ht]
\includegraphics[scale=0.5]{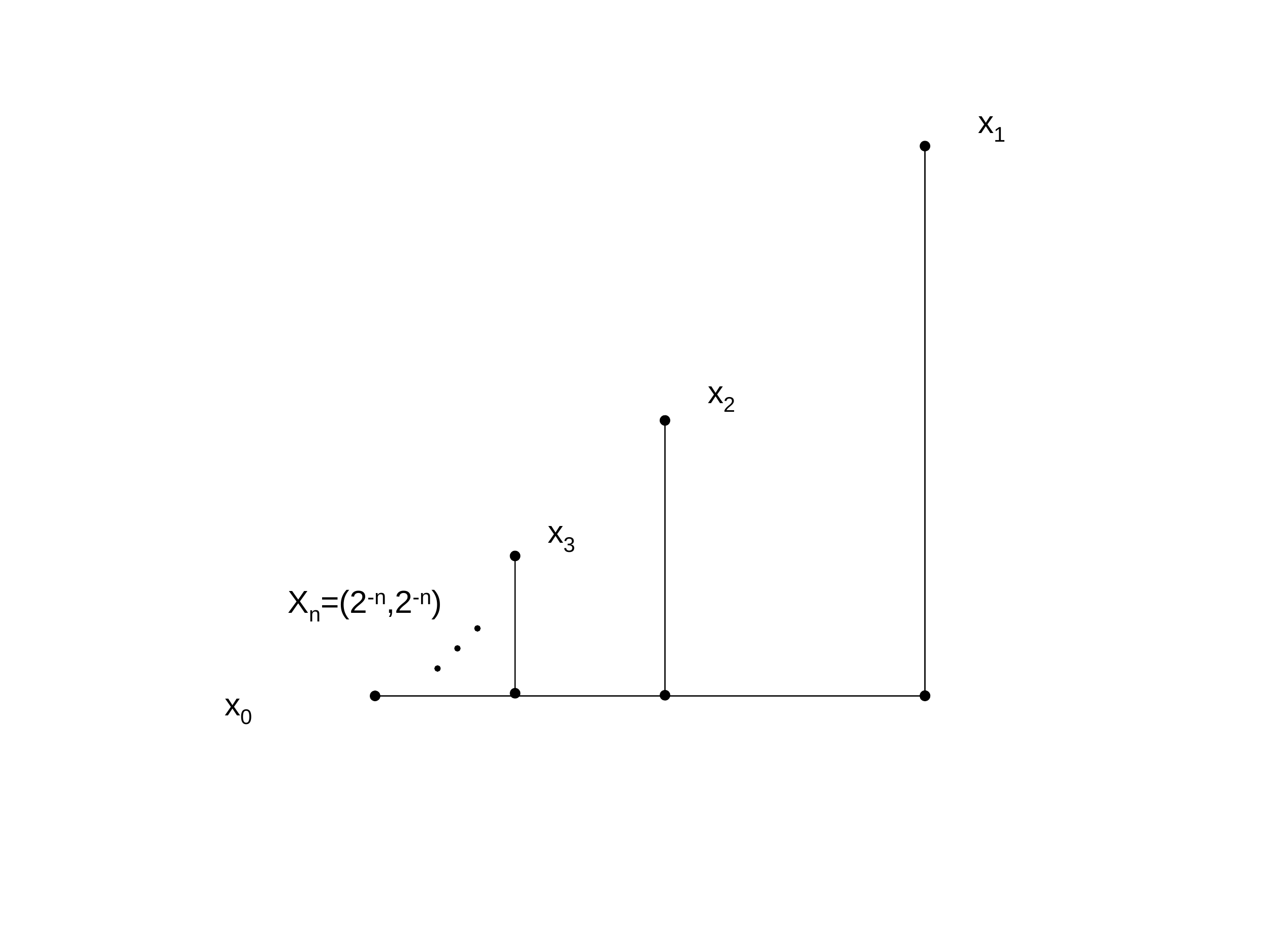}
\caption{Compact length space which is not topologically robust.}\label{fig.3}
\end{figure}

This is a compact connected length space.
Then for every $\varepsilon>0$ let $n_0$ be such that
$\frac{1}{2^{n_0}}<\frac{\varepsilon}{2}$ and consider the points
$x=(0,0)$ and $y=(\frac{1}{2^{n_0}},\frac{1}{2^{n_0}})$. Clearly
$\bar{B}(x,\varepsilon)=\bar{B}(y,\varepsilon)$ and $p_\varepsilon$ fails
to be injective.

In fact, for every $\varepsilon>0$, $p_\varepsilon(X)$ is not homeomorphic to $X$.
For every $0<\varepsilon<diam(X)$ and $\forall n$ such that
$\frac{2}{2^n}\leq \varepsilon$ the geodesic segment
$[(2^{-n},2^{-n}),(2^{-n},0)]$ is identified with the geodesic segment
$[(0,0),(2^{-n},0)]$, and it is readily seen that $p_\varepsilon(X)$ is a finite tree.

Notice that, in this example, $M=2$, $m=1$, its extrema are the points 
$\Phi^{-1}(2)=\cup_{n\in \mathbb{N}\cup 0}\{(2^{-n},2^{-n})\}\cup \{0,0\}$ and its center is the point $\Phi^{-1}(1)=(1,0)$.

\vspace{0.5cm}

We take the following definition from \cite{B-H}. See Definition 3.26 and, explicitly on page 119.

\begin{definicion} Given $r>0$, a metric space (X,d) is said to be
\emph{r-uniquely geodesic} if for every pair of points $x,y$ with
$d(x,y)<r$ there is a unique geodesic segment joining $x$ to $y$.
\end{definicion}

\begin{definicion} We define X to be \textbf{r-perfectly geodesic} if it is r-uniquely
geodesic and for any three points $x_1,x_2,x_3$ with
$d(x_i,x_j)<r$ if the geodesic segments $[x_1,x_2]$ and $[x_1,x_3]$ have a
common non-trivial interval then one of the geodesic segments is contained in the other.
\end{definicion}

\begin{nota} When this property holds, for any $x_1,x_2,x_3$ with
$d(x_i,x_j)<r$, $[x_1,x_2] \cap [x_1,x_3]$ equals $\{x_1\},
[x_1,x_2]$ or $[x_1,x_3]$.
\end{nota}

\begin{nota} \label{quotient} If $X$ is a length space, $p_\varepsilon: X \to
p_\varepsilon(X)$ is continuous. Since $X$ is compact and $2^X_H$ is
Hausdorff (and then, so it is $p_\varepsilon(X)$ for any $\varepsilon$),
$p_\varepsilon: X \to p_\varepsilon(X)$ is a quotient map and
$p_\varepsilon (X)$ is homeomorphic to the quotient space $X/_\sim$
where two points are related $x\sim y$ if and only if
$p_\varepsilon(x)=p_\varepsilon(y)$. Thus, it is trivial that when
$p_\varepsilon$ is injective, $X$ is homeomorphic to $p_\varepsilon(X)$.
\end{nota}

\begin{teorema}\label{flujo homeom.geod} If $(X,d)$ is a r-perfectly geodesic compact length space,
then $(X,d)$ is topologically robust.
\end{teorema}

\begin{proof} As we saw in Remark \ref{quotient}, it suffices to
check the injectivity. So, let us see that there exists some
$\varepsilon_0>0$ such that for every pair of points $x,y$ of $X$,
$\bar{B}(x,\varepsilon) \neq \bar{B}(y,\varepsilon) \ \forall \varepsilon
\leq \varepsilon_0$.

Let $\varepsilon:=\frac{r}{4}$ and assume $r<<diam(X)$. If $d(x,y)>\varepsilon$
it is trivial, so let us suppose
$d(x,y)\leq \varepsilon \leq \varepsilon_0$. Consider any point $z\in X$ such that
$d(z,x),d(z,y)>r$ (remember that $r<<diam(X)$) and, since $X$ is a
length space, let $z_0$ be the point in the geodesic segment $[z,x]$ such that
$d(z_0,x)=\frac{r}{2}$. Clearly, $\varepsilon \leq
d(z_0,x),d(z_0,y)<r$. Now the geodesic segments $[z_0,x]$ and $[z_0,y]$
are unique and $[z_0,x]\cap [z_0,y]$ may be $[z_0,x]$, $[z_0,y]$
or $\{z_0\}$.

Case 1: if $[z_0,x]\cap [z_0,y]$ is $[z_0,x]$. (If this
intersection was equal to $[z_0,y]$ it would be analogous). The
geodesic segment $[z_0,y]$ is isometric to a subinterval of the real
line $[0,d(z_0,y)]$; let us denote such isometry from the
subinterval to the geodesic as $f_{[z_0,y]}$. In this case, there
exists some $t\in [0,d(z_0,y)]$ such that $f_{[z_0,y]}(t)=x$ and
since $d(z_0,x)\geq \varepsilon$ then clearly $t\geq \varepsilon$.

\begin{picture}(1,1)(-100,50)
    \put(0,0){\line(1,0){90}}
    \put(0,0){\circle*{4}}
    \put(-20,0){$z_0$}
    \put(90,0){\circle*{4}}
    \put(90,-10){$y$}
    \put(70,0){\circle*{4}}
    \put(65,-10){$x$}
    \put(40,0){\circle*{3}}
    \put(10,10){$f_{[z_0,y]}(t-\varepsilon)$}
\end{picture}

\vspace{3cm}

Clearly, the point $f_{[z_0,y]}(t-\varepsilon)$ is in
$\bar{B}(x,\varepsilon)$ but it is not in $\bar{B}(y,\varepsilon)$ since
the distance to $y$ through this geodesic is obviously $\varepsilon +
d(x,y)$ and this distance is less than $r$ which means that this
is the unique geodesic from this point to $y$.

Case 2: if $[z_0,x]\cap [z_0,y]$ is $\{z_0\}$. Assume
$d(z_0,x)\leq d(z_0,y)$ and consider $z_1$ the point in the
geodesic segment $[z_0,x]$ such that $d(z_1,x)=\varepsilon$. Clearly,
$d(z_1,y)\geq d(z_0,y)- d(z_0,z_1)\geq d(z_0,x)- d(z_0,z_1)=
\varepsilon$. If $d(z_1,y)\neq \varepsilon$ then $z_1 \in
\bar{B}(x,\varepsilon)\backslash \bar{B}(y,\varepsilon)$ and hence
$\bar{B}(x,\varepsilon) \neq \bar{B}(x,\varepsilon)$.

Suppose then that $d(z_1,y)=\varepsilon$. Hemce, $[z_0,z_1]\cup
[z_1,y]$ has the same length of $[z_0,x]$ and, since $d(z_0,y)\geq
d(z_0,x)$, it defines a geodesic from $z_0$ to $y$ whose
intersection with $[z_0,x]$ is $[z_0,z_1]$ and this contradicts
the fact that there is a unique geodesic from $z_0$ to $y$ which
must be the one which intersected $[z_0,x]$ just in $\{z_0\}$.
\end{proof}

Given a connected Riemannian manifold, there is a natural length metric induced by the length piecewise continuously differentiable paths. See \cite[I.3]{B-H}.

\begin{cor}\label{riemann} If $(X,d)$ is a compact connected Riemannian manifold with its natural length metric, then $(X,d)$ is topologically robust.
\end{cor}

\begin{proof} It is a basic result on Riemannian geometry that
every point of a Riemannian manifold lives at the center of a
convex ball such that any two points in that ball are joined by a
unique geodesic segment contained in the ball. Since it is
compact, through the Lebesgue number we can find a global radius
$r$.
\end{proof}

The following corolary follows from the fact that geodesics in a space of curvature $\geq k$ (with $k$ an
arbitrary real number) do not branch. The definitions can be found in \cite{Bu-Bu} where this statement is left as an exercise, see 10.1.2.

\begin{cor}\label{curvat} If $(X,d)$ is a locally uniquely geodesic compact length space
of curvature bounded bellow, then $(X,d)$ is topologically robust.
\end{cor}

This condition on the geodesics is sufficient but it is not
necessary. There are important groups of spaces for which this map
would be an embedding and they are not necessarily r-perfectly
geodesic, for example trees or
finite polyhedra with a length metric. These will be further
referred as {\it polyhedral spaces}. Let us begin by endowing an n-simplex $\Delta^n$ with vertices
$x_0,\cdots , x_n$ with a length metric. Let us consider it as a subspace of
$\mathbb{R}^{n+1}$ with $x_{i-1}=\varepsilon_0 e_i$
where $e_1,\cdots , e_{n+1}$ is the canonical basis of
$\mathbb{R}^{n+1}$. Therefore, we will say that the n-simplex is endowed the euclidean metric. Note that if the vertices are $x_{i-1}=\varepsilon_0 e_i$, the edges have length $\sqrt{2} \cdot
\varepsilon_0$.

\begin{lema} Let $\Delta^n$ be a n-simplex endowed with the euclidean
metric $d$, with vertices $x_0,\cdots , x_n$ and suppose length
$\varepsilon$ for the edges. For every pair of points given in
barycentric coordinates $x:=\lambda_0 x_0 + \cdots + \lambda_n
x_n$ and $x':=\lambda'_0 x_0 + \cdots + \lambda'_n x_n$ the
distance between those points is given by the formula \[d(x,x'):=
\frac{\varepsilon}{\sqrt{2}} \sqrt{\sum_{i=0}^n (\lambda_i -
\lambda'_i)^2}\]
\end{lema}

\begin{proof} In order to get length $\varepsilon$ on the edges, since the euclidean distance
between two vertices is $\sqrt{2} \cdot \varepsilon_0$, it suffices to take
$\varepsilon_0=\frac{\varepsilon}{\sqrt{2}}$ and measure the euclidean distance in the n-simplex as a subset of
$\mathbb{R}^{n+1}$.
\end{proof}

Consider any finite simplicial complex $K$. If we consider the
geometric realization $|K|$, this finite polyhedron can be metrized
with a length metric $d$ in a natural way. Set each simplex
isometric to a euclidean one and assume length $\sqrt{2}$ on the
edges for simplicity. Now for any two points $x,y$ in $|K|$,
$d(x,y)$ will be defined as the greatest lower bound of the length
of PL paths joining them. 
(It is immediate to see that
if $K$ is a finite simplicial complex this is a metric and the
metric topology is the same of $|K|$).

Thus, $d$ will be referred to as a \textbf{polyhedral metric} and $|K|$ endowed with the metric $d$, $|K|_d$, as a
\textbf{finite polyhedral space}.

\begin{nota}\label{coord} If we have a finite polyhedron K with vertices $x_0,\cdots,
x_n$, for any point $x \in K$ we can represent it in barycentric
coordinates as $\sum_{i=0}^n \lambda_i x_i$ where if $x$ belongs
to a simplex with vertices $x_0,\cdots ,x_k$ then $\lambda_i=0 \
\forall  i\neq 0,\cdots k$. The distance between two points  $\sum_{i=0}^n \lambda_i x_i$ and
$\sum_{i=0}^n \lambda'_i x_i$ in the
same simplex, measured in the euclidean metric of that simplex, is
then $\sqrt{\sum_{i=0}^n (\lambda_i-\lambda'_i)^2}$.
\end{nota}

\begin{lema}\label{mismo simplex} For any two points $x,x' \in \Delta$ with $\Delta$ any
simplex of $K$, the distance $d(x,x')$ in $|K|_d$ is the distance
in $\Delta$ when considered as isometric to an euclidean simplex
of diameter $\sqrt{2}$.
\end{lema}

\begin{proof} Consider the points in barycentric coordinates
$x:=\lambda_0 x_0 + \cdots + \lambda_n x_n$ and $x':=\lambda'_0
x_0 + \cdots + \lambda'_n x_n$ with $x_0,\cdots ,x_n$ all the
vertices of $K$ as we saw in remark \ref{coord}. The euclidean
distance in the simplex is $d_0=\sqrt{\sum_{i=0}^n (\lambda_i -
\lambda'_i)^2}$. Consider now a PL path joining $x$ and $x'$ which
is a finite union of linear paths joining $x:=y_0$ to $y_1$, $y_1$
to $y_2$, $\cdots y_{k-1}$ to $y_k:=x'$ where $y_{j-1},y_j$ belong
to the same simplex $\forall j=1,k$. Let us denote
$\beta_1^j,\cdots ,\beta_n^j $ the barycentric coordinates of
$y_j$ (Note that $\beta_i^0=\lambda_i$ and
$\beta_i^k=\lambda'_i$). Then, the length of this path may be
computed as

\[l_0:=\sqrt{
\sum_{i=0}^n(\beta_i^0-\beta_i^1)^2}+\sqrt{
\sum_{i=0}^n(\beta_i^1-\beta_i^2)^2}+ \cdots +\sqrt{
\sum_{i=0}^n(\beta_i^{k-1}-\beta_i^k)^2}\]

and by Minkowski's
inequality,

\[l_0\geq
\sqrt{\sum_{i=0}^n(\beta_i^0-\beta_i^k)^2}=d_0.\]

Then $d_0$ is a
lower bound of the length of these paths finishing the
proof.
\end{proof}

\begin{teorema}\label{flujo homeom.pol} Let $|K|_d$ be a finite polyhedral space. Then $|K|_d$ is topologically robust.
\end{teorema}

\begin{proof} Let $n:=max\{dim(\Delta_i) \ | \ \Delta_i \in K\}$ and
$\varepsilon< \frac{1}{n}$.

As we mentioned before, it suffices to check that it is injective. Consider $x,x'$ any two points in $|K|_d$. If
$d(x,x')>\varepsilon$ then obviously $\bar{B}(x,\varepsilon) \neq
\bar{B}(x',\varepsilon)$. If $d(x,x')\leq \varepsilon$ and there is some
simplex $\Delta_i$ such that $x,x'\in \Delta_i$, then restricting
to this simplex (where the restricted metric is the euclidean one)
it is clear that the closed balls do not coincide.

Then, the remaining case is when the points $x,x'$ are in different
simplices $\Delta, \Delta'$ and $d(x,x')<\varepsilon$.

Suppose that these points are, in barycentric coordinates, $x=\lambda_0 x_0 +
\cdots \lambda_n x_n$ and $x'=\lambda'_0 x_0 + \cdots \lambda'_n
x_n$, in the representation we saw at Remark \ref{coord}. Let us
assume that $\sum_{i=0}^n (\lambda_i)^2\leq \sum_{i=0}^n
(\lambda'_i)^2$. We can choose a vertex (there is no loss of generality if we consider it $x_0$) such that $x_0\in \Delta \backslash \Delta'$ with
$\lambda_0>0$. (This can be done because otherwise $x$ would be in
$\Delta'$).

Now we claim that $d(x_0,x)<d(x_0,x')$. Clearly
\[d(x_0,x)=\sqrt{(1-\lambda_0)^2 + \sum_{i=1}^n \lambda_i^2}.\]

To compute $d(x_0,x')$ consider as in lemma \ref{mismo simplex} any
sequence of points $x_0:=y_0,y_1,\cdots ,y_k:=x'$ with
$y_{j-1},y_j$ in the same simplex and $\beta_1^j,\cdots ,\beta_n^j
$ the barycentric coordinates of $y_j$ ($\beta_0^0=1$,
$\beta_i^0=0 \ \forall \, i\neq 0$ and $\beta_i^k=\lambda'_i$).

The length of this PL-path is
\[l_0=\sqrt{\sum_{i=0}^n(\beta_i^0-\beta_i^1)^2}+\sqrt{
\sum_{i=0}^n(\beta_i^1-\beta_i^2)^2}+ \cdots +\sqrt{
\sum_{i=0}^n(\beta_i^{k-1}-\beta_i^k)^2}\]
and by Minkowski's
inequality, \[l_0\geq
\sqrt{\sum_{i=0}^n(\beta_i^0-\beta_i^k)^2}=\sqrt{(1-\lambda'_0)^2
+ \sum_{i=1}^n (\lambda'_i)^2}\] but $x_0 \not \in \Delta'$ and
therefore $\lambda'_0=0$.

Finally, the assumption that $\sum_{i=0}^n (\lambda_i)^2\leq
\sum_{i=0}^n (\lambda'_i)^2$ (together with $\lambda_0>0$ and $\lambda'_0=0$) implies
that $d(x_0,x)^2=(1-\lambda_0)^2 + \sum_{i=1}^n (\lambda_i)^2 < 1+
\sum_{i=0}^n (\lambda_i)^2 - \lambda_0 \leq 1+ \sum_{i=0}^n
(\lambda'_i)^2 - \lambda_0 \leq (l_0)^2 - \lambda_0$. Thus
\[\sqrt{d(x_0,x)^2+\lambda_0}<l_0 \mbox{ and }
d(x_0,x)<\sqrt{d(x_0,x)^2+\lambda_0}\leq d(x_0,x').\] 

Then, let $z_0$ be the point in the geodesic segment $[x_0,x]$  such that
$d(z_0,x)=\varepsilon$. Since $d(x_0,x)< d(x_0,x')$,
$d(z_0,x)<d(z_0,x')$ and hence $z_0\not \in \bar{B}(x',\varepsilon)$. Therefore, the closed balls do not coincide.
\end{proof}

\begin{quest} Is it true that for every 2-dimensional connected Riemannian manifold with its usual length metric any two balls with the same radius and different center coincide if and only if they are the whole space?
\end{quest}

\begin{quote} If this were true, then the dynamical cone of any Riemannian manifold would be exactly the topological cone.
\end{quote}

Nevertheless, being topologically robust is a strongly
geometric geometric condition. We can give an example of a compact length space
$(X,d)$ such that considering an equivalent metric up to
bi-lipschitz homeomorphism $(X,d')$ the condition holds for $(X,d)$ and not for
$(X,d')$.
\begin{ejp} Consider $X=\{(0,y)\, | \, 0\leq y \leq 1\}\cup
\{(x,0)\, | \, 0\leq x \leq 1\} \cup \{(x,y)\, | \, 0\leq x \leq y= 2^{-i},
i>0\}\cup \{(x,y)\, | \, 0\leq y \leq x = 2^{-i}, i>0\}$. See Figure \ref{fig.4}
\end{ejp}

\begin{figure}[ht]
\includegraphics[scale=0.5]{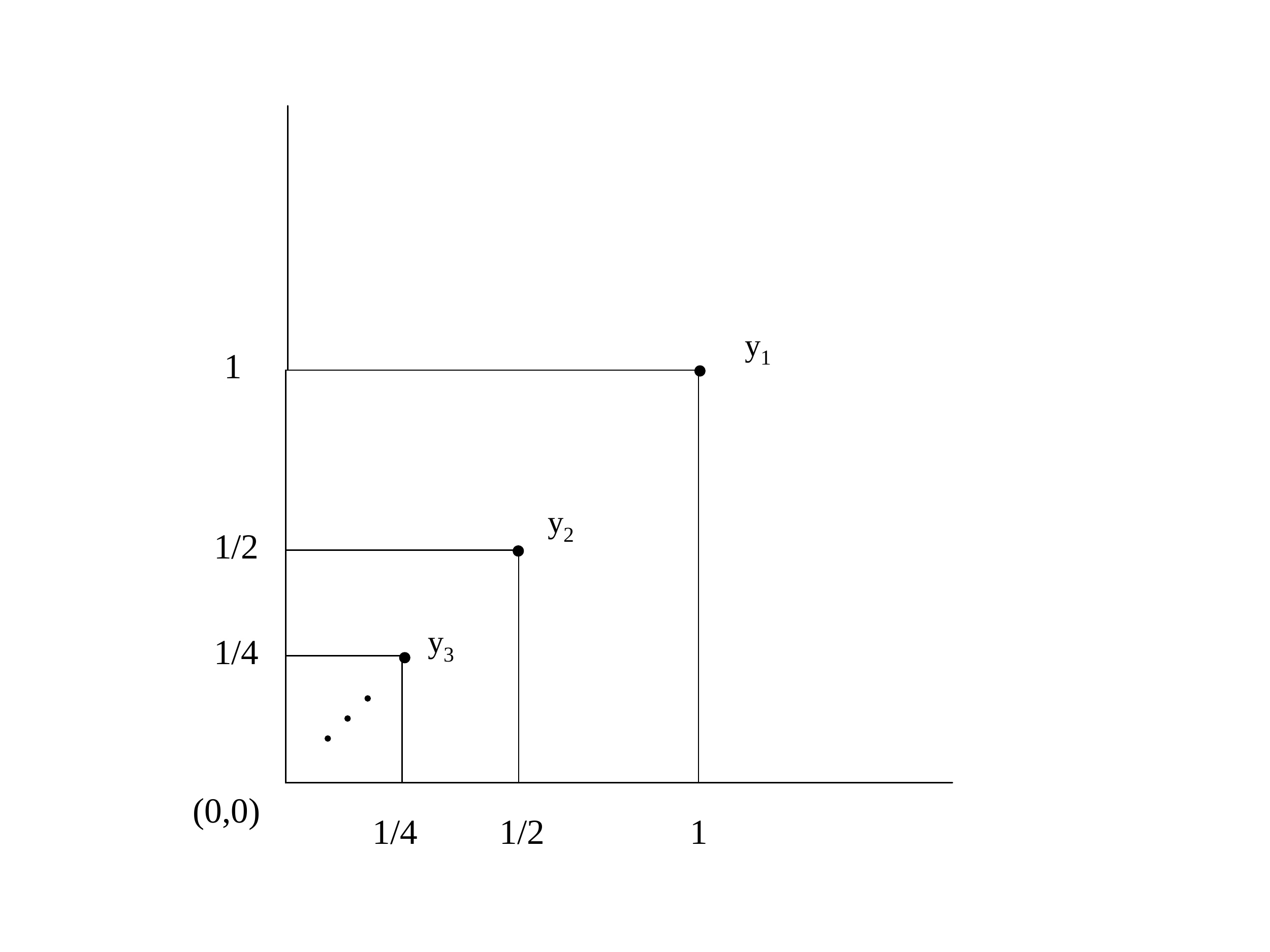}
\caption{The space is not topologically robust.}\label{fig.4}
\end{figure}

In $X$ with the natural length metric the balls centered at
$x=(0,0)$ and $y_i=(2^{-i},2^{-i})$ and radius $\varepsilon_i=2\cdot
2^{-i}$ coincide. Consider the space $X'=\{(0,y)\, | \, 0\leq y \leq
1\}\cup \{(x,0)\, | \, 0\leq x \leq 1\} \cup (\cup_{i>0} \{(x,y)\, | \,  0\leq
x,y \mbox{ and } x+y=2^{-i}\})$ with the natural length metric. See Figure \ref{fig.5}.

\begin{figure}[ht]
\includegraphics[scale=0.5]{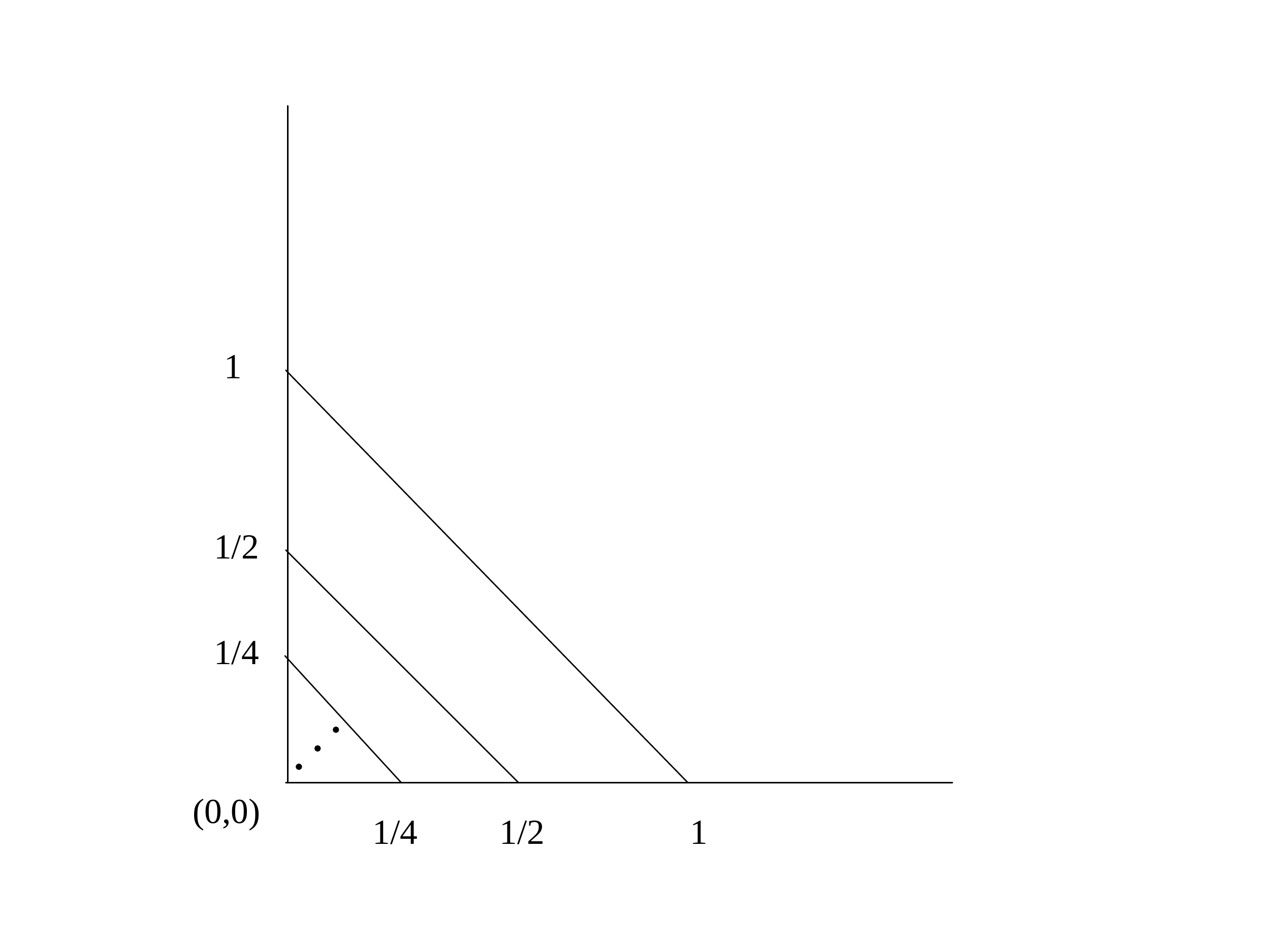}
\caption{The space is topologically robust.}\label{fig.5}
\end{figure}

Note that $X'$ can be obtained from $X$ up to isometry, just by
the following change in the distance: for any pair of points
$z,z'$ in $\{(x,y)|0\leq x \leq y= 2^{-i}\}\cup \{(x,y)|0\leq y
\leq x = 2^{-i}\}$ for some $i>0$,
$d'(z,z')=\frac{d(z,z')}{\sqrt{2}}$. Now in $(X,d')$, for any
$\varepsilon\leq \frac{1}{2}$ it may be easily checked that two balls
of radius $\varepsilon$ and different centers do not coincide.

\section{The semiflow for finite metric graphs.}

By a graph, we mean a 1-dimensional simplicial complex. 

A \emph{finite metric graph} is a connected finite graph endowed with the usual length metric where every edge has length 1. In this section we prove that if $X$ is a finite metric graph, then being a finite graph is a positively invariant property. This is, if $X$ is a finite metric graph, then $p_\varepsilon(X)$ is (topologically) a finite graph for every $\varepsilon>0$. Moreover, the graph goes through a finite number of topological types before colapsing to a point.

The underlying idea is that, in order to study geometric properties of the initial metric space, this tool preserves more information that other alternatives as Whitney levels which need not be even 1-dimensional. See \cite{I-N} and \cite{N2} for further information about Whitney levels.

\vspace{0.5cm}

\paragraph{\textbf{Finite trees}}$\\$

By a tree, we refer to a 1-dimensional simply connected simplicial complex. A \emph{rooted tree} $(T,v)$ consists of a tree $T$ and a point $v\in T$, called  the \emph{root}. If $c$ is any point of the rooted tree
$(T,v)$, the \emph{subtree of $(T,v)$ determined by c} is
\[T_c=\{x\in T \, | \ c\in [v,x]\}.\]

\begin{lema} Let $x,y$ be two points in a finite tree $T$
such that $p_\varepsilon(x)=p_\varepsilon(y)$. Then, for any pair of points $z,z'$ in the geodesic segment
$[x,y]$ with $z=tx + (1-t)y$ and
$z'=(1-t)x + ty$ for any $t$,  $p_\varepsilon(z)=p_\varepsilon(z')$.
\end{lema}

\begin{proof} Consider $v\in [x,y]$ the middle point of the
geodesic, $d(x,v)=d(v,y)=\frac{1}{2}d(x,y)$. Let $(T,v)$ be the
rooted tree and $T_x,T_y$ the corresponding subtrees. Since $\bar{B}(x,\varepsilon)=\bar{B}(y,\varepsilon)$
it is easy to see that $T_x \subset \bar{B}(y,\varepsilon)$ and $T_y
\subset \bar{B}(x,\varepsilon)$. In fact, $\forall  z \in [x,v], z
\neq v \quad T_z \subset \bar{B}(y,\varepsilon)$ and conversely
$\forall  z \in [v,y], z \neq v \quad T_z \subset
\bar{B}(x,\varepsilon)$. Else, suppose $z \in [x,v], z\neq v$ with
$T_z \not \subset \bar{B}(y,\varepsilon)$. Then, there is a point
$p\in T_z$ such that $d(p,y)=\varepsilon + \delta$ and, by the
properties of the length metric, we can assume this $\delta>0$ to be 
smaller than $d(z,v)$. If this is so, then $d(p,x)\leq
d(p,z)+d(z,x) < d(p,z) + d(z,y) - 2\delta = \varepsilon -\delta$ and
hence, $p \in \bar{B}(x,\varepsilon)$ which contradicts the fact that
the balls centered in $x$ and $y$ coincide.

Moreover, $\forall  z \in [x,v], z \neq v$ and $\forall  z' \in
[v,y], z \neq v \quad T_z \subset \bar{B}(z',\varepsilon)$ and
$T_{z'} \subset \bar{B}(z,\varepsilon)$ with the same argument since
$d(z,z')<d(z,y),d(x,z')$. Thus, if both points $z,z'$ are at the
same distance to the root $v$, then their balls necessarily coincide.
\end{proof}

Also, the following lemmas are clear from the proof.

\begin{lema} Consider any pair of points in a finite tree, $x,y \in T$,
such that $p_\varepsilon(x)=T=p_\varepsilon(y)$. Then, $p_\varepsilon(z)=T$ for every $z\in [x,y]$. In particular,
the set $T_\varepsilon':=\{x\in T \, | \, p_\varepsilon(x)=T \}$ is a (possibly empty) subtree. \hfill$\square$ 
\end{lema}

\begin{lema} If two points $x,y$ are such that $p_\varepsilon(x)=p_\varepsilon(y)$
and the middle point $v$ of the geodesic has order two in the tree,
then $p_\varepsilon(z)=T$ for every $z\in [x,y]$.\hfill$\square$ 
\end{lema}

\begin{lema} There are no cycles in $p_\varepsilon(T)$.\hfill$\square$ 
\end{lema}

\begin{prop} If $T$ is a finite tree, then $p_\varepsilon(T)$ is a finite
tree for every $\varepsilon$.
\end{prop}

\begin{proof} For any $\varepsilon>0$ consider $T_\varepsilon'=\{x\in T \, | \, p_\varepsilon(x)=T \}$. As we mentioned before, $T_\varepsilon'$ is a subtree (it may be the empty set). We can
construct $p_\varepsilon(T)=T/_\sim$ (see \ref{quotient}) in two steps: first $T/_{T'_\varepsilon}$ identifying all
the points in $T'_\varepsilon$. This is obviously homeomorphic to a tree. The
second step is to identify every other pair of points with the same projection but, since there is a finite number of points of order $>2$
and the tree is locally finite, we are identifying a finite number
of geodesic segments without generating any cycle. Therefore, we obtain a
tree.
\end{proof}

\paragraph{\textbf{Finite graphs}}$\\$

In this paragraph the space $X$ will always be a non-degenerated finite metric graph.

Let us recall that $B^c(x,\varepsilon)=\bar{B}(x,\varepsilon)$ and
$\partial \bar{B}(x,\varepsilon)=\partial
B(x,\varepsilon)=S(x,\varepsilon)=\{z\in X \ | \ d(z,x)=\varepsilon \}$.

\begin{lema}\label{finite points}  For any point $x \in X$ and any
$\varepsilon>0$, $\partial B(x,\varepsilon)$ consists at most of a
finite number of points.
\end{lema}

\begin{proof} Let $x\in |e|$ for some edge $e$ with vertices $v_0,v_1$ (if $x$ is
a vertex suppose it is $v_0$), and let $d_0=d(x,v_0)$ and
$d_1=d(x,v_1)=1-d_0$ (let us assume, without loss of generality, that
$d_0\leq d_1$, and $d_0=0$ if $x$ is a vertex). For every vertex
$v_i\in X$ the distance to $v_0$ is a positive integer $n_i$ and
the distance to $v_1$ is $m_i$ (obviously $|n_i-m_i|\leq 1$). Hence
the distance from $v_i$ to $x$ is either $d_0+n_i$ if $n_i\leq m_i$
or $d_1+m_i$ if $m_i<n_i$. For any point in any edge $y\in [v,v']$
the distance $d(x,y)$ is $min\{d(y,v)+d(v,x),d(y,v')+d(v'x)\}$.

If $\varepsilon\leq d_0$ then $\bar{B}(x,\varepsilon)$ is contained in the edge and
$\partial B(x,\varepsilon)$ consists of two points. If $d_0<\varepsilon \leq d_1\leq
1$, then the part of the edge between $x$ and $v_1$ contains one border
point and the rest of them are one at each edge adjacent to $v_0$.
If $\varepsilon>d_1\geq d_0$, let $d_0'=\varepsilon-d_0 -
[\varepsilon-d_0]$ (where $[t]$ is the integer part of $t$) and $d_1'=\varepsilon-d_1
- [\varepsilon-d_1]$. For every $z\in X$ such that $d(z,x)=\varepsilon$,
since there is a path realizing the distance in the graph, there
must be a vertex $w\in X$ such that $d(z,w)=d_0'$ if the geodesic segments contain $v_0$ or $d(z,w)=d_1'$ if the geodesic segments contain $v_1$.

Since the number of vertices and edges is finite, the number of
points at distance $d_0'$ or $d_1'$ from any vertex is finite and
so it is the number of points in $S(x,\varepsilon)=\partial
B(x,\varepsilon)$.
\end{proof}

Obviously, $\partial B(x,\varepsilon)\neq \emptyset$ if and only if
$\forall \varepsilon'<\varepsilon \quad \bar{B}(x,\varepsilon')\neq X$.

\begin{lema}\label{finite points 2} $\forall \varepsilon>0$ there is a
finite number of points $z_1,\cdots,z_n$ for which $\partial
\bar{B}(z_i,\varepsilon)$ contains a vertex or the middle point of an
edge.
\end{lema}

\begin{proof} Let $x\in X$ and let us denote for $x\in
[v_0,v_1]$, $d_0=d(x,v_0) \leq d_1=d(x,v_1)=1-d_0$.

Suppose $\varepsilon>1$. For any point $x\in X$, let $z$ be a vertex or a middle point of an edge in the border of the ball of radius $\varepsilon$. Since the distance form $z$ to any vertex or middle point is a multiple of $\frac{1}{2}$, considering the geodesic segment from $x$ to $z$ (which
contains $v_0$ or $v_1$), then $|\varepsilon-d_0|$ or $|\varepsilon-d_1|$ is
a multiple of $\frac{1}{2}$. There are at most two points holding
this at each edge. Since the graph is finite, there is a finite
number of points for which $\partial \bar{B}(z_i,\varepsilon)$ contains a vertex or a middle point of an edge.

If $\varepsilon \leq 1$ it suffices to consider the points at distance
$\varepsilon$ or $|\varepsilon-\frac{1}{2}|$ from the vertices which is a finite set.\end{proof}

\begin{lema}\label{dist unica} For any pair of points $x,y \in X$
there exists some $\delta>0$ such that $\forall z \in B(x,\delta)\backslash\{x\},
\ d(z,y)\neq d(x,y)$. Moreover, we can choose $\delta>0$ such that
each connected component $C_i$ of $B(x,\delta)\backslash \{x\}$ is
contained in some edge and $\forall z\in C_i$ $d(z,y)=d(x,y) +
d(x,z)$ or $d(z,y)=d(x,y) - d(x,z)$.
\end{lema}

\begin{proof} Let us divide the proof in two cases.

First when $x$ is not a vertex. Let $x\in |e|$ for some edge $e$
with vertices $v_0,v_1$ and $d_0=d(x,v_0)$, $d_1=d(x,v_1)=1-d_0$
with $d_0\leq d_1$. If $y\in [v_0,v_1]$ let
$\delta<\min\{d_0,d_1,d(x,y)\}$ and the result is obvious. So, let us assume that $y\not \in [v_0,v_1]$. If there is a
geodesic segment $[x,y]$ containing $v_i, \ i\in \{0,1\}$ and $z\in
(x,v_i)$ (the points in $|e|$ between $x$ and $v_i$) then
$d(z,y)=d(x,y)-d(x,z)$. If no geodesic segment  $[x,y]$
contains $v_i, \ i\in \{0,1\}$ it means that
$d_i+d(v_i,y)>\varepsilon$. Then, let $0<2\delta<d_i+d(v_i,y)-
\varepsilon$. It is immediate to see that $\forall z\in (x,v_i)$
such that $d(x,z)<\delta$, any geodesic segment $[z,y]$ still contains
the opposite vertex and $d(z,y)=d(x,z)+d(x,y)$.

If $x$ is a vertex of the graph, then let $w_1,\cdots, w_n$ all the adjacent
vertices. If $y\in [x,w_i]$ for some $i$ it suffices to take
$\delta<d(x,y)$. If $y \not \in [x,w_i]$ for every $i$, let
$w_1,\cdots,w_k$ those $w_i$ for which $d(w_i,y)=d(x,y)-1$. Then, for any $i>k$, either
$d(w_i,y)=d(x,y)$ or $d(w_i,y)=d(x,y)+1.$ Let $\delta<\frac{1}{2}$.

If $z\in
(x,w_i)$ with $i\leq k$ then $d(z,y)=d(x,y)-d(x,z)$. If $z\in
(x,w_i)$ with $i> k$ and $d(z,x)\leq \delta$, then any geodesic segment $[z,y]$ contains $x$ and some $w_i$ with $i\leq k$. Hence, $d(z,y)=d(x,y)+ d(x,z)$.
\end{proof}

Let $\mathcal{V}$ be the set of vertices in $X$ and $\mathcal{M}$ the
set of middle points of edges.

For every $\varepsilon>0$ let $A_\varepsilon:=\{x\in X \ | \ x\not
\in \mathcal{V}, x\not \in \mathcal{M}, \partial B(x,\varepsilon)\cap
\mathcal{V}=\emptyset,  \partial B(x,\varepsilon)\cap
\mathcal{M}=\emptyset \mbox{ and } \bar{B}(x,\varepsilon)\neq X\}$.

\begin{prop} \label{finite points proj} $\forall \varepsilon>0$
$p_\varepsilon(X)\backslash p_\varepsilon (A_\varepsilon)$ is a finite
number of points.
\end{prop}

\begin{proof} $X\backslash A_\varepsilon$ consists of all the points in
$X$ for which the ball of radius $\varepsilon$ is the total space
together with a finite number of points by \ref{finite points 2}
and because the graph is finite. $p_\varepsilon(X)\backslash
p_\varepsilon (A_\varepsilon)$ consist of the projection of that finite
number of points together with the total space if there is such a
ball.
\end{proof}

\begin{nota} \label{ditta dist borde} For every $x\in A_\varepsilon$,
if $\varepsilon>d_1$, $\varepsilon=d_0+k+d_0'=d_1+k'+d_1'$ and
$0<d_0'\neq \frac{1}{2}$. Then $d_1'=d_0-d_1 +k-k'
+d_0'=d_0-1+d_0+k-k'+d_0'=2d_0+d_0'+k''$ with $k''$ some integer,
and since $2d_0$ is not an integer $d_1'\neq d_0'$.
\end{nota}

\begin{lema}\label{short path} Let $x\in A_\varepsilon$ with
$d_1<\varepsilon=d_0+k+d_0'=d_1+k'+d_1'$ and $y\in \partial
\bar{B}(x,\varepsilon) \cap cl(X\backslash \bar{B}(x,\varepsilon))$.
Then there is an edge $e'=[w,w']$ such that $y\in e'$,
$d(y,w)=d(x,y)-d(y,w)$ and $d(y,w')>d(x,y)-d(y,w')$. Moreover, either
$d(w,y)=d'_0$ and any geodesic segment from $x$ to $y$ is
$[x,v_0]\cup [v_0,w]\cup [w,y]$ 
with lengths $d_0,k,d'_0$
respectively or $d(w,y)=d'_1$ and it is $[x,v_1]\cup [v_1,w]\cup
[w,y]$ with lengths $d_1,k',d'_1$.
\end{lema}

\begin{proof} Since $y\in \partial
\bar{B}(x,\varepsilon)$ there is a geodesic segment $[x,y]$ of
length $\varepsilon$ which contains one (and only one) of
the vertices, let us consider it $w$, of $e'$ ($y$ is not a vertex
because $x\in A_\varepsilon$). Since $y \in cl(X\backslash
\bar{B}(x,\varepsilon))$, then $d(y,w')>d(x,y)-d(x,w')$. Otherwise
$[y,w']\cup [w',x]$  would
be also a path of length $\varepsilon$ and $|e'|$ would be contained in
$\bar{B}(x,\varepsilon)$. This would be a contradiction because it would
make $d(y,X\backslash \bar{B}(x,\varepsilon))\geq min\{d'_0,d'_1\}>0$
and $y \not \in cl(X\backslash \bar{B}(x,\varepsilon))$.

The distance between any two vertices is an integer and $d_0<d_1$.
Then $d(x,w)$ can be of the type $d_0+k$ or $d_1+k'$. The first
case occurs if and only if $d(w,y)=d'_0$ and any geodesic segment $[x,y]$ would be $[x,v_0]\cup [v_0,w]\cup [w,y]$  with
$d(v_0,w)=k$ and the second one occurs if and only if
$d(w,y)=d'_1$ and the geodesic segment would be $[x,v_1]\cup
[v_1,w]\cup [w,y]$ with $d(v_1,w)=k'$.
\end{proof}

\begin{nota}\label{dist unica 3} If we apply lemma \ref{dist unica} we obtain a ball
about this point $y\in e'$ and both connected components $C_0\subset (w_i,y)$
where $\forall z \in C_0 \quad d(z,x)=d(x,y)-d(z,y)$ and
$C_1\subset (y,w'_i)$ where $\forall z \in C_1 \quad
d(z,x)=d(x,y)+d(z,y)$.
\end{nota}

\begin{lema} \label{subset border} For every $x\in A_\varepsilon$ and $\{y_1,\cdots,y_n\}=
\partial \bar{B}(x,\varepsilon) \cap cl(X\backslash
\bar{B}(x,\varepsilon))$ there are two disjoint subsets  $\{y_1,\cdots,y_k\}$ and $\{y_{k+1},\cdots,y_n\}$
so that for any geodesic segment $\gamma_i$ from $x$ to $y_i$
$[v_0,x]\cap \gamma_i\neq \{x\}$ for $i\leq k$ and $[v_1,x]\cap
\gamma_i\neq \{x\}$ for $i>k$. In particular, if
$d_1<\varepsilon=d_0+k+d_0'=d_1+k'+d_1'$, then for every $i\leq k$
there is a vertex $w_i$ so that
$d(w_i,y_i)=d'_0$ and $d(x,w_i)=d_0+k$ and for every $i> k$
there is a vertex $w_i$ so that
$d(w_i,y_i)=d'_1$ and $d(x,w_i)=d_1+k'$
\end{lema}

\begin{proof} The case when $\varepsilon\leq d_1$ is trivial. If
$\varepsilon>d_1$ the proof follows easily from lemma \ref{short path}.
\end{proof}

The interesting case comes when we consider lemma \ref{dist unica}
applied to a point in $A_\varepsilon$ and the points in $\partial
\bar{B}(x,\varepsilon) \cap cl(X\backslash \bar{B}(x,\varepsilon))$
using this partition.

\begin{lema} \label{dist unica 2} Let $x\in A_\varepsilon$ and
$\{y_1,\cdots,y_n\}= \partial \bar{B}(x,\varepsilon) \cap
cl(X\backslash \bar{B}(x,\varepsilon))$ with the partition defined
in Lemma \ref{subset border}. Then there is some $\delta>0$ such
that $B(x,\delta)\backslash \{x\}$ is contained in an edge and has
two connected components $C_{0,\delta}\subset
(v_0,x),C_{1,\delta}\subset (v_1,x)$, and $\forall z\in
C_{0,\delta}$, $d(z,y_i)=d(x,y_i)-d(x,z) \ \forall i\leq k$,
$d(z,y_i)=d(x,y_i)+d(x,z) \ \forall i> k$ and $\forall z\in
C_{1,\delta}$, $d(z,y_i)=d(x,y_i)-d(x,z) \ \forall i> k$ and
$d(z,y_i)=d(x,y_i)+d(x,z) \ \forall i\leq k$.
\end{lema}

\begin{proof} The cases where $\varepsilon \leq d_1<1$ are quite
trivial and it suffices to take $\delta<d_0,\varepsilon$. Let us
study the case when $\varepsilon>d_1$. Let $x\in |e|=[v_0,v_1]$,
$d_0=d(x,v_0)<d(x,v_1)=d_1$ and $\varepsilon=d_0+k+d_0'=d_1+k'+d_1'$
(remember that, since $x\in A_\varepsilon$, $d_0',d_1'\neq
0,\frac{1}{2}$).

By lemma \ref{short path}, every $y_i$ is contained in an edge
$[w_i,w'_i]$ with $d(w_i,y_i)=d'_0$ $\forall i\leq k$ and
$d(w_i,y_i)=d'_1$ $\forall i>k$.

Let $0<2 \delta_i<d(w'_i,x)+ d(w'_i,y_i)- \varepsilon$. Then, if
$\delta_0=\min\{\delta_i\}$, $\forall z\in B(x,\delta_0)$ any
geodesic segment $[z,y_i]$ contains $w_i$ and not
$w'_i$.

If $\delta_1<d_0$ there are two connected components in
$B(x,\delta_1)\backslash \{x\}$: $C_{0,\delta_1}\subset (v_0,x)$
contained in the part of the edge $e$ between $v_0$ and $x$, and
$C_{1,\delta_1}\subset (v_1,x)$.

Finally, let $0<2\delta'_i<d(y_i,v_1)+d_1-\varepsilon$ for $i\leq k$
and $0<2\delta'_i<d(y_i,v_0)+d_0-\varepsilon$ for $i> k$, and
$\delta_2=\min\{\delta'_i,\delta_1\}$. Then $\forall z\in
B(x,\delta_2)$ the geodesic segment $[z,y_i]$ still contains
$v_0$ if $i\leq k$ and $v_1$ if $i>k$.

Then define $\delta=\min\{\delta_0,\delta_2\}$. If $z\in
C_{0,\delta}$, $d(z,y_i)=d(z,v_0)+d(v_0,w_i)+ d(w_i,y)=
d(x,y_i)-d(x,z) \ \forall i\leq k$ and
$d(z,y_i)=d(z,v_1)+d(v_1,w_i)+ d(w_i,y)=d(x,y_i)+d(x,z) \ \forall
i> k$ and if $z\in C_{1,\delta}$, $d(z,y_i)=d(z,v_1)+d(v_1,w'_i)+
d(w'_i,y)=d(x,y_i)-d(x,z) \ \forall i> k$ and
$d(z,y_i)=d(z,v_0)+d(v_0,w'_i)+ d(w'_i,y)=d(x,y_i)+d(x,z) \
\forall i\leq k$.
\end{proof}

\begin{lema}\label{Lema: ciclos}  Suppose that $x$ is not a vertex nor a middle point and let $y' \in \partial  \bar{B}(x,\varepsilon)\backslash cl(X\backslash \bar{B}(x,\varepsilon))$ so that $y'$ is not a vertex. Let $[v_0,v_1]$ and $[u_0,u_1]$ be two edges such that $x\in [v_0,v_1]$ and $y'\in [u_0,u_1]$. Then, $[v_0,v_1]$ and $[u_0,u_1]$ are part of a minimal
cycle of length $2\varepsilon\in \mathbb{N}$ composed by
two geodesic segments from $x$ to $y'$. Moreover, if $2\varepsilon$ is
even, then $d_0=d_1'$ and $d_1=d_0'$ and if $2\varepsilon$ is odd
then $d_1'=d_0+\frac{1}{2}$ and $d_0'=d_1-\frac{1}{2}$.
\end{lema}

\begin{proof} Let $\varepsilon=d_0 +k
+d'_0=d_1+k' +d'_1$. Let us assume, with no loss of generality, that $d_0\leq d_1$,  $d(y',u_0)=d'_0>0$,
$d(y',u_1)=d'_1=1-d'_0>0$, $d(x,u_0)=d_0 +k$ and $d(x,u_1)=d_1+ k'$.
Since $y'\not \in cl(X\backslash \bar{B}(x,\varepsilon))$, there exist two geodesic segments $\gamma_0,\gamma_1$  with
length $\varepsilon$ from $x$ to $y'$, $\gamma_0$ containing $v_0$ and
$u_0$, and $\gamma_1$ containing $v_1$ and $u_1$.

Consider the restriction $\gamma_0'$ of $\gamma_0$ joining $v_0$ and $u_0$
and the restriction $\gamma_1'$ of $\gamma_1$ joining $v_1$ and
$u_1$. If they are disjoint we are done. Otherwise, there would be a
common vertex $z\subset \gamma_0'\cap \gamma_1'$ and
$d_0+n_1=d(v_0,z)\neq d(v_1,z)=d_1+n_2$ (see Figure \ref{fig_grafo}).

\begin{figure}[ht]
\includegraphics[scale=0.5]{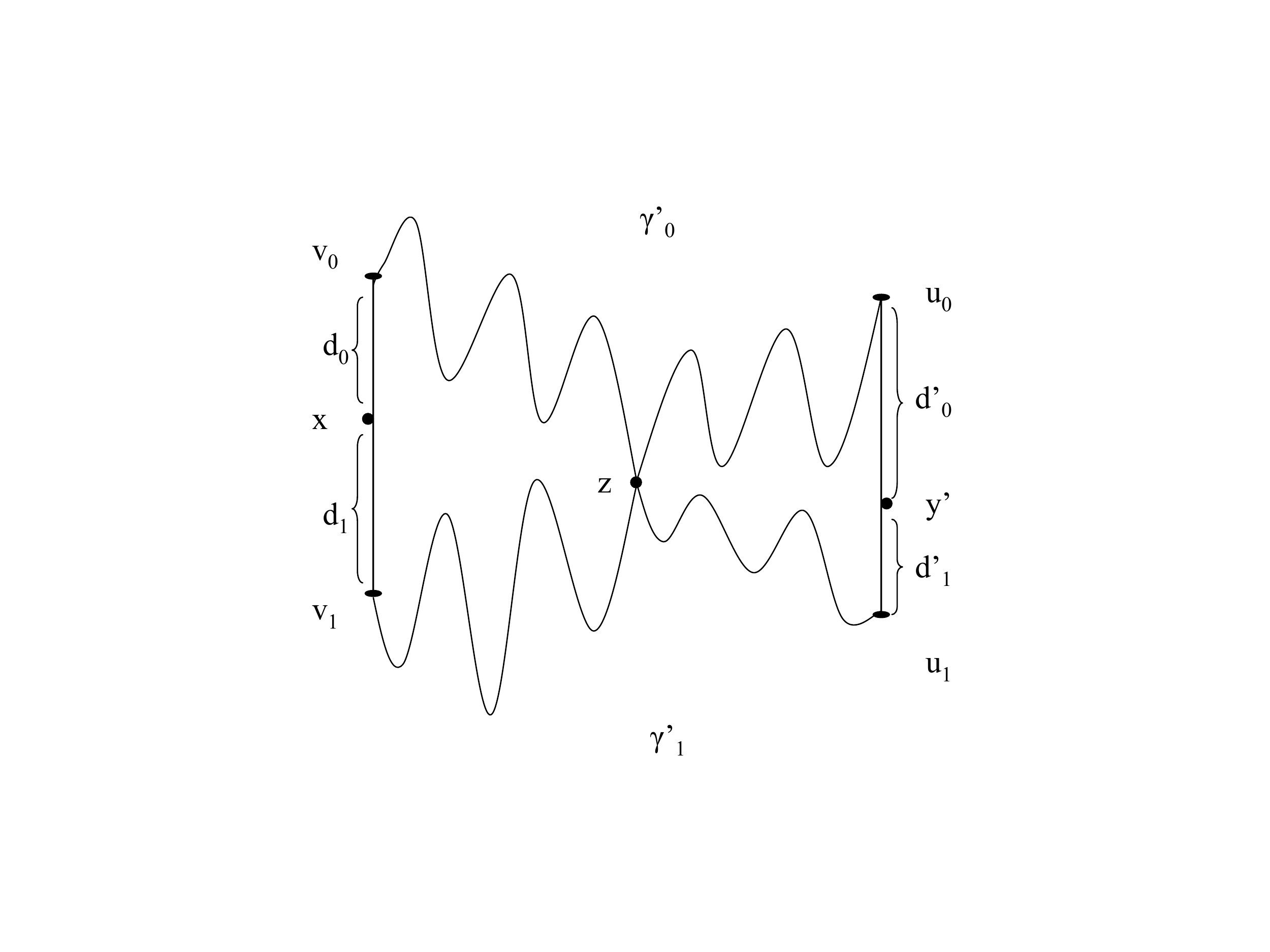}
\caption{In the conditions of the Lemma, there is a contradicion if the geodesic segments $\gamma'_0,\gamma'_1$ intersect.} \label{fig_grafo}
\end{figure}

If $d(v_0,z)\leq d(v_1,z)$ the vertices $u_0,u_1$ hold that
$d(u_0,v_0)\leq d(u_0,v_1)$ and $d(u_1,v_0)\leq d(u_1,v_1)$ but
this last inequality is not possible because there is a geodesic segment from $x$ to $u_1$ containing $v_1$ (restriction of $\gamma_1$)
and since $d_1>d_0$ we would get
$d(x,u_1)=d_1+d(u_1,v_1)>d_0+d(u_1,v_0)\geq d(x,u_1)$.

Otherwise $d(v_1,z)\leq d(v_0,z)-1$ and then $d(v_1,u_0)\leq
d(v_0,u_0)-1$ and $d(v_1,u_1)\leq d(v_0,u_1)-1$. Then $[x,v_1]\cup
[v_1,u_0] \cup [u_0,y']$ defines a path of length $d_1+
d(v_1,u_0) +d'_0\leq d_1+d(v_0,u_0)-1 +d'_0$ and since $d_1<
d_0+1$ this yields a path from $x$ to $y'$ shorter than
$\gamma_0$, this is, $d(x,y')<\varepsilon$ which is a contradiction.
\end{proof}

Let us define \begin{equation}\label{Definicion: K} K_{x,\varepsilon}:=\{x\in |e| \, | \, |e|\subset B(x,\varepsilon)\},\end{equation} the set of points in $X$
which belong to some edge entirely contained in the open ball
$B(x,\varepsilon)$. Notice that, since $X$ is a finite graph, $K_{x,\varepsilon}$ is compact.

\begin{lema}\label{balls 2} Let $x\in A_\varepsilon$, $\{y_1,\cdots,
y_n\}=\partial \bar{B}(x,\varepsilon)\cap cl(X\backslash
\bar{B}(x,\varepsilon))$ the partition defined in Lemma \ref{subset border}. 
Then there exists $ \delta>0$ such that,
$\forall z \in C_{0,\delta}$ \quad \[\bar{B}(z,\varepsilon)=
\bar{B}(x,\varepsilon) \ \cup \ \{\cup_{i=1}^k \bar{B}(y_i,d(z,x))\} \
\backslash \ \cup_{i=k+1}^{n} B(y_i,d(z,x))\] and \ $\forall z \in
C_{1,\delta}$ \quad \[\bar{B}(z,\varepsilon)= \bar{B}(x,\varepsilon) \
\cup \ \{\cup_{i=k+1}^n \bar{B}(y_i,d(z,x))\} \ \backslash \
\cup_{i=1}^{k} B(y_i,d(z,x))\]  
In particular we can take $\delta$
small enough so that each ball $\bar{B}(y_i,d(z,x))$ is contained in
some edge and it is disjoint from the other balls.
\end{lema}

\begin{proof} If $\varepsilon\leq d_1$ the lemma is immediate taking
$\delta<\min\{d_0,\varepsilon\}$. Let us suppose $\varepsilon> d_1$.

Let $\delta_0<d'_0,d'_1,1-d'_0,1-d'_1$.

\begin{quote} The first two bounds are
quite redundant with the next one, $\delta_1$, and may be
eliminated from the proof but to justify how they follow from the
other one is less clear than explicitly asking for them. The last two
are unnecessary for the first part of this result but, later on,
it will be useful to make sure that each ball $B(y_i,d(z,x))$ is contained in one edge.
\end{quote}

Consider $\delta_1>0$ so that lemma \ref{dist unica 2} is
satisfied.

Let $\varepsilon'=\max_{a \in
K_{x,\varepsilon}}\{d(a,x)\}$. Since
$K_{x,\varepsilon}$ is compact and it is contained in the open ball, $0\leq \varepsilon'<\varepsilon$.
Let $0<\delta_2<\varepsilon-\varepsilon'$ then $\forall z\in
B(x,\delta_1)$, $K_{x,\varepsilon}\subset B(z,\varepsilon)$.

Thus, we only have to care about edges $|e_i|=[w_i,w'_i]$ containing
$\{y_i\}$ for $i=1,n$ and those $e'_1,\cdots, e'_r$ such that
$|e'_j|\subset \bar{B}(x,\varepsilon)$ but there is a
point $y'_j$ of the border of the ball in them. Let us start with one of these $y'_j\in
|e'_j|$ with $1\leq j\leq r$.

Since $x\in A_\varepsilon$, $y'=y'_j$ is an interior point of $e'_j=[u_0,u_1]$. By Lemma \ref{Lema: ciclos},
any border point $y'$ in an edge entirely
contained in the closed ball is in a cycle of length $2\varepsilon\in
\mathbb{N}$ given by two geodesic segments from $x$ to $y'$. Moreover
if $2\varepsilon$ is even, then $k=k'$ and hence $d_0=d_1'$ and
$d_1=d_0'$. If $2\varepsilon$ is odd then $k'=k+1$ and $d_0+
d_0'+1=d_1+d_1'$ (other thing is not possible because $d_0<d_1$),
and since $d_1=1-d_0$ and in this situation $d_0'=1-d_1'$ it
follows immediately that $d_1'=d_0+\frac{1}{2}$ and
$d_0'=d_1-\frac{1}{2}$.

The important fact here for these points $y'_i$ is that if we
consider $\delta_3<d_0<d_1$, for any point $z\in B(x,\delta_3)$
the closed ball $\bar{B}(z,\varepsilon)$ also contains any cycle of
length $2\varepsilon$ containing the edge $[v_0,v_1]$ and there would
not be any difference between the balls centered in $x$ and in $z$
in those edges $|e'_i|$.

Hence if $0<\delta<\delta_1,\delta_2,\delta_3$ the unique
difference between those balls would be in those edges $[w_i,w'_i]$
which contain the points $\{y_1,\cdots,y_n\}$ (not outside
$\bar{B}(x,\varepsilon)$ either because $\delta<1-d'_0,1-d'_1$).
(Note that two of these points $y_i,y_j$ may be in the same edge
if $2d'_0<1$, $2d'_1<1$ or $d'_0+d'_1<1$. This would mean in the
notation that $w_i=w'_j$ and $w'_i=w_j$ and doesn't lead to any
contradiction).

Let $0<\delta_4=\frac{1}{2} min_{i\neq j}\{d(y_i,y_j)\}$. If
$0<\delta<\delta_4$, for any pair of border points $y_i,y_j$
$\bar{B}(y_i,\delta)\cap \bar{B}(y_j,\delta)=\emptyset$.

Thus finally, let
$0<\delta<\min\{\delta_0,\delta_1,\delta_2,\delta_3,\delta_4\}$.

We are going to prove the equality for the case $z\in
C_{0,\delta}$. If $z\in C_{1,\delta}$ is analogous:

$\bar{B}(z,\varepsilon)= \bar{B}(x,\varepsilon) \cup_{1}^k
\{\bar{B}(y_i,d(x,z))\}\backslash \cup_{i=k+1}^{n}
\{B(y_i,d(x,z))\}$.

As we said, $K_{x,\varepsilon}$ and $\{e'_1,\cdots,e'_r\}$ are both
contained in $\bar{B}(z,\varepsilon)$ and in $\bar{B}(x,\varepsilon)
\cup_{1}^k \{\bar{B}(y_i,d(z,x))\}\backslash \cup_{i=k+1}^{n}
\{B(y_i,d(z,x))\}$ since $\delta<\delta_1<d'_0,d'_1$ and we are
not removing any of those points with the balls $B(y_i,d(z,x))$
with $i>k$ because $\delta<d'_1$. To prove the equality it remains
to see what happens on $|e_i|$.

The balls $B(y_i,\delta)$ are all disjoint and contained in an
edge so we can study what happens around each border point
independently.

If $i\leq k$ then $d(z,y_i)=\varepsilon-d(z,x)$ and hence the ball
$\bar{B}(z,\varepsilon)$ includes $[w_i,y_i]$ and also around $y_i$
exactly a ball $\bar{B}(y_i,d(x,z)$. If $i>k$
$d(z,y_i)=\varepsilon+d(z,x)$ and $d(z,x)<\delta<d'_1$ implies that
$\bar{B}(z,\varepsilon)\cap [w_i,y_i]=[w_i,y_i]\backslash
B(y_i,d(x,z))$.
\end{proof}

\begin{lema} \label{balls} Let $\bar{B}(x,\varepsilon)=X$ and let $\{y_1,\cdots,
y_n\}=\partial \bar{B}(x,\varepsilon)$. Let $\delta_1>0$ be such that
$B(x,\delta_1)$ holds lemma \ref{dist unica} for every border
point $y_i$. For any component $C_i\in B(x,\delta_1)\backslash
\{x\}$ consider $\{y_1,\cdots,y_k\}$ those border points such that
$\forall z\in C_i \quad d(z,y_{i})=d(x,y_{i})+d(x,z)$. Then, there
exists some $\delta< \delta_1$ such that $\forall z\in C_i\cap
B(x,\delta)$, $\bar{B}(z,\varepsilon)=X \backslash \cup_{i=1}^{k}
\{B(y_i,d(x,z))\}$.
\end{lema}

\begin{proof} If $\varepsilon\leq 1$ then $X$ consists of two vertices
and a single edge or two edges and $x$ is a common vertex. In both
cases the proof is immediate. Let us suppose that there are at
least two edges and $\varepsilon> 1$.

Let $x\in [v_0,v_1]$ with $0\leq d_0=d(x,v_0)<1$ and
$d_1=d(x,v_1)=1-d_0$. Let $k=[\varepsilon-d_0]$, $k'=[\varepsilon-d_1]$,
$d'_0=\varepsilon-k-d_0$ and $d'_1=\varepsilon-k'-d'_0$.

Let $\varepsilon'=\max_{a \in K_{x,\varepsilon}}\{d(a,x)\}$. Then $0\leq
\varepsilon'<\varepsilon$ since $K_{x,\varepsilon}$ is compact and it is
contained in the open ball. Let $0<\delta_2<\varepsilon-\varepsilon'$
then $\forall z\in B(x,\delta_2)$, $K_{x,\varepsilon}\subset
B(z,\varepsilon)$.

If $\{y_1,\cdots,y_k\}\neq \emptyset$ and
$\delta<\min\{\delta_1,\delta_2\}$ we only have to check those
edges containing a border point $y_i$ with $i\leq k$. Any geodesic segment
$[x,y_i]$ has length $\varepsilon$ and $\forall z\in
C_i$, any geodesic segment $[z,y_i]$ has length
$\varepsilon+d(z,x)$.

Thus, $d(z,x)=\varepsilon+d(z,x)$ implies that $\forall z'\in
B(y_i,d(z,x))$, $d(z',z)> \varepsilon$.
Then $B(y_i,d(x,z))\cap \bar{B}(z,\varepsilon)=\emptyset$ and the
ball $\bar{B}(z,\varepsilon)\subset X \backslash \cup_{i=1}^{k}
\{B(y_i,d(x,z))\}$.

On the other hand, consider any point $z'\in X \backslash
\cup_{i=1}^{k} \{B(y_i,d(x,z))\}$. If $z'\in K_{x,\varepsilon}$, then, since $d(x,z)<\delta, \delta_2$,  $z'\in
\bar{B}(z,\varepsilon)$. If $z'\not \in K_{x,\varepsilon}$, then $z'\in [w_i,y]$ with $[w_i,w_i']$ some edge
such that $y_i\in [w_i,w'_i]$. Therefore, $d(z,z')\leq
d(x,y)+d(x,z)-d(z',y) =\varepsilon +d(x,z)-d(z',y)$. But, since
$z'\not \in B(y_i,d(x,z))$, $d(z',y_i)\geq d(x,z)$ and therefore,
$d(z,z')\leq \varepsilon$ and $z'\in \bar{B}(z,\varepsilon)$.

Thus, $\bar{B}(z,\varepsilon) = X \backslash \cup_{i=1}^{k}
\{B(y_i,d(x,z))\}$.
\end{proof}

\begin{cor} \label{isometry 2} Let $\bar{B}(x,\varepsilon)=X$ and let $\emptyset \neq \{y_1,\cdots,
y_n\}=\partial \bar{B}(x,\varepsilon)$. Then there is some $\delta>0$
such that the restriction $p_\varepsilon|_{C_{i,\delta}\cup \{x\}}:C_{i,\delta}\cup \{x\}\to
p_\varepsilon(C_{i,\delta}\cup \{x\})$ is an isometry when
$p_\varepsilon(C_{i,\delta})\neq \{X\}$.
\end{cor}

\begin{proof} By the properties of the length metric $\forall
z,z'\in X \quad d_H(\bar{B}(z,\varepsilon),\bar{B}(z',\varepsilon))\leq
d(z,z')$. Suppose $\delta>0$ holding Lemma \ref{balls}. If
$p_\varepsilon(C_{i,\delta})\neq \{X\}$ then, by Lemma \ref{dist
unica}, there is some $y_j$ such that $\forall z\in C_{i,\delta}$
$d(z,y_j)=d(z,x)+d(x,y_j)$.

As we have just shown
$\bar{B}(z,\varepsilon)\cap B(y_j,d(z,x))=\emptyset$.
Let $z'\in (z,x]$. $\bar{B}(z',\varepsilon)\cap S(y_j,d(z',x))\neq
\emptyset$ since $d(z',y_j)=\varepsilon+d(z',x)$.

Thus, any point $p\in \bar{B}(z',\varepsilon)\cap S(y_j,d(z',x))$
holds that $d(p,\bar{B}(z,\varepsilon))\geq d(z,x)-d(z',x)=d(z,z')$
and hence $d_H(\bar{B}(z,\varepsilon),\bar{B}(z',\varepsilon))\geq
d(z,z')$ which proves the equality.
\end{proof}

\begin{prop}\label{isometry} For any $x\in X$ such that
$\bar{B}(x,\varepsilon)\neq X$, if $x$ is not a vertex nor a middle
point of an edge then there exists $\delta>0$ such that the
restriction $p_\varepsilon |_{B(x,\delta)} :B(x,\delta) \to p_\varepsilon
(B(x,\delta))$ is an isometry.
\end{prop}

\begin{proof} If $\varepsilon \leq d_1$ it is immediate.

Otherwise, let $y\in \partial \bar{B}(x,\varepsilon)\cap cl(X\backslash
\bar{B}(x,\varepsilon))$ and $\delta_1>0$ such that Lemma
\ref{dist unica} holds for $x,y$. Since $x$ is not a vertex,
$B(x,\delta_1)\backslash \{x\}$ decomposes in two connected
components $C_{0,\delta_1}$,$C_{1,\delta_1}$ and, since $x$ is not
a middle point, $d_0=d(x,v_0)\neq d(x,v_1)=d_1$. Therefore, $y$ is contained
in an edge $[w,w']$ with $d(w,x)=d(x,y)-d(w,y)$ (there is a
geodesic segment $\gamma=[x,y]$ with $w\in \gamma$) and
$d(w',x)>d(x,y)-d(w',y)$ (every geodesic segment $[x,y]$ must contain
$w$ because $y\in cl(X\backslash
\bar{B}(x,\varepsilon))$).

Since $d_0\neq d_1$, $0<|d_0-d_1|<1$. As we saw in Lemma \ref{finite
points}, the distance between vertices of the graph is an integer
which now implies that $d(x,v_0)+d(v_0,w)\neq d(x,v_1)+d(v_1,w)$.
Then suppose $\varepsilon=d(x,v_0)+d(v_0,w)+
d(w,y)<d(x,v_1)+d(v_1,w)+d(w,y)=\varepsilon'$ and any geodesic segment
$\gamma$ contains $v_0$. See Figure \ref{grafolema2}.

\begin{figure}[ht]
\includegraphics[scale=0.5]{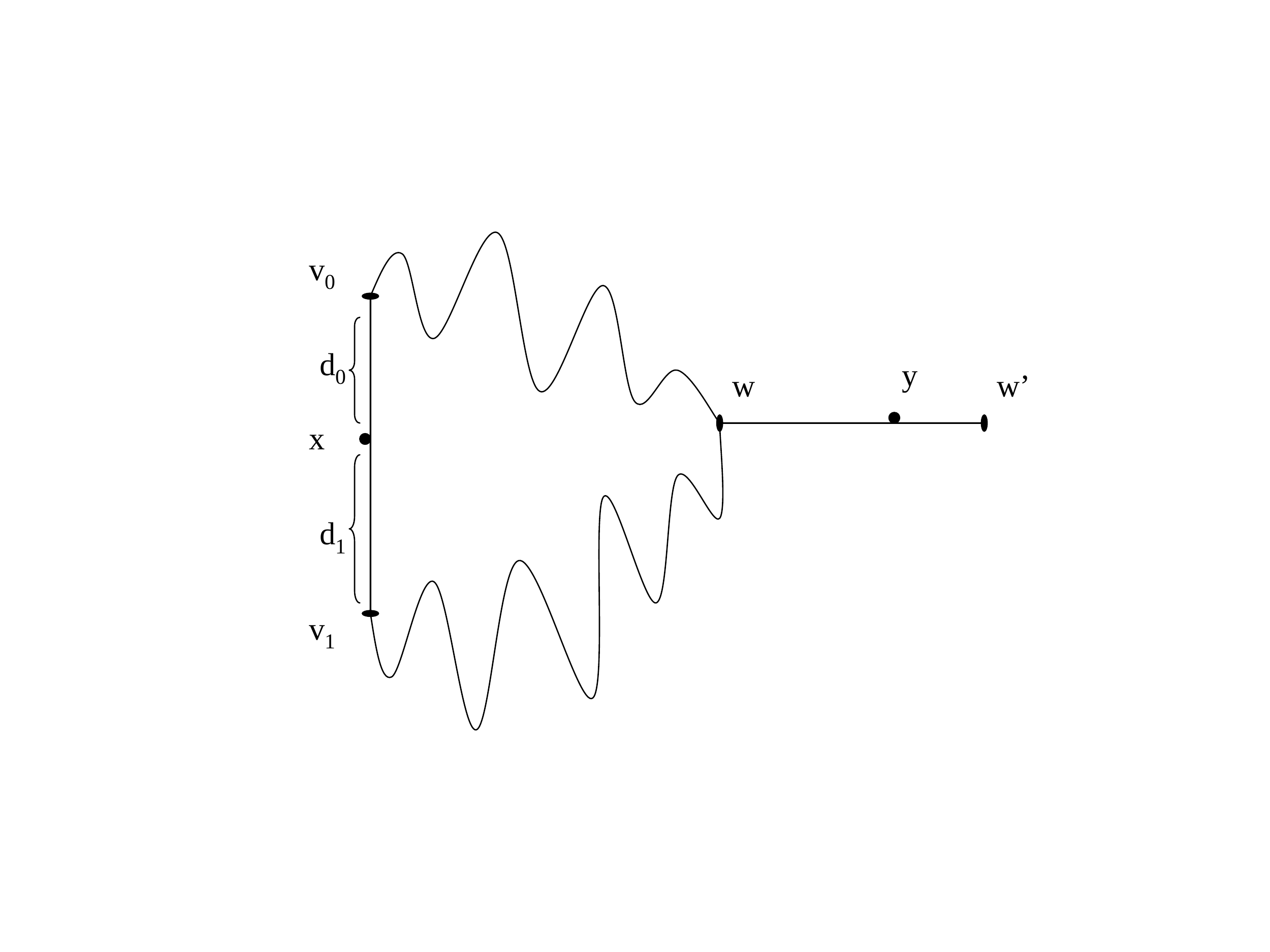}
\caption{Every geodesic segment $[x,y]$ contains the same vertex $v_0$ or $v_1$.}\label{grafolema2}
\end{figure}

If we apply Lemma \ref{dist unica} on $y,v_0$ obtaining a ball
around $y$, then there is a component $D_{1,\delta_2}$
contained in $(y,w']$ and such that
$d(z',v_0)=d(z',y)+d(y,v_0) \forall z'\in D_{1,\delta_2}$. If it
were such that $d(z',v_0)=d(y,v_0)-d(z',y)$ there would be a path
of length $\varepsilon-d_0$ given by $[v_0,z']\cup [z',y]$ with
$z'\in (y,w')$. The image of this path necessarily contains the vertex $w'$ and
we will arrive to a contradiction because
$d(y,w')+d(w',v_0)+d(v_0,x)\geq d(y,w')+d(w',x)>
d(y,w)+d(w,v_0)+d(v_0,x)=\varepsilon$ and hence $d(y,w')+d(w',v_0)>
d(y,w)+d(w,v_0)=\varepsilon-d_0$ and there is no geodesic segment
$[v_0,y]$ containing $w'$.

Let $\delta<\min\{\delta_1,\delta_2,\frac{1}{4}
(\varepsilon-\varepsilon')\}$. Then $\forall z\in B(x,\delta)$ and
$\forall z'\in B(y,\delta)$, $z\in [v_0,v_1]$ and any geodesic segment
$\gamma=[z,z']$ contains $v_0$ since otherwise we
would have a path form $x$ to $y$ across $v_1$ with length $\leq
\varepsilon +2\delta <\varepsilon'$.

Then, if $z\in C_{0,\delta}$, $D_{0,\delta}\subset
\bar{B}(z,\varepsilon)$ and $\bar{B}(z,\varepsilon)\cap
D_{1,\delta}=\bar{B}(y,d(z,x))\cap D_{1,\delta}$. If $z\in
C_{1,\delta}$ $D_{1,\delta} \cap \bar{B}(z,\varepsilon)=\emptyset$
and $\bar{B}(z,\varepsilon)\cap D_{0,\delta}=D_{0,\delta}\backslash
B(y,d(z,x))$. Thus, it is immediate to check that $\forall
z_1,z_2\in B(x,\delta)$ $d_H(\bar{B}
(z_1,\varepsilon),\bar{B}(z_2,\varepsilon))\geq d(z,z')$.
\end{proof}

This holds, in particular, for any $x\in A_\varepsilon$. Nevertheless,
the projection $p_\varepsilon$ need not be open.

\begin{ejp} Consider the graph in Figure \ref{ejpabierta} and the projection with $\varepsilon=2+\frac{3}{4}$.
\end{ejp}

\begin{figure}[ht]
\includegraphics[scale=0.5]{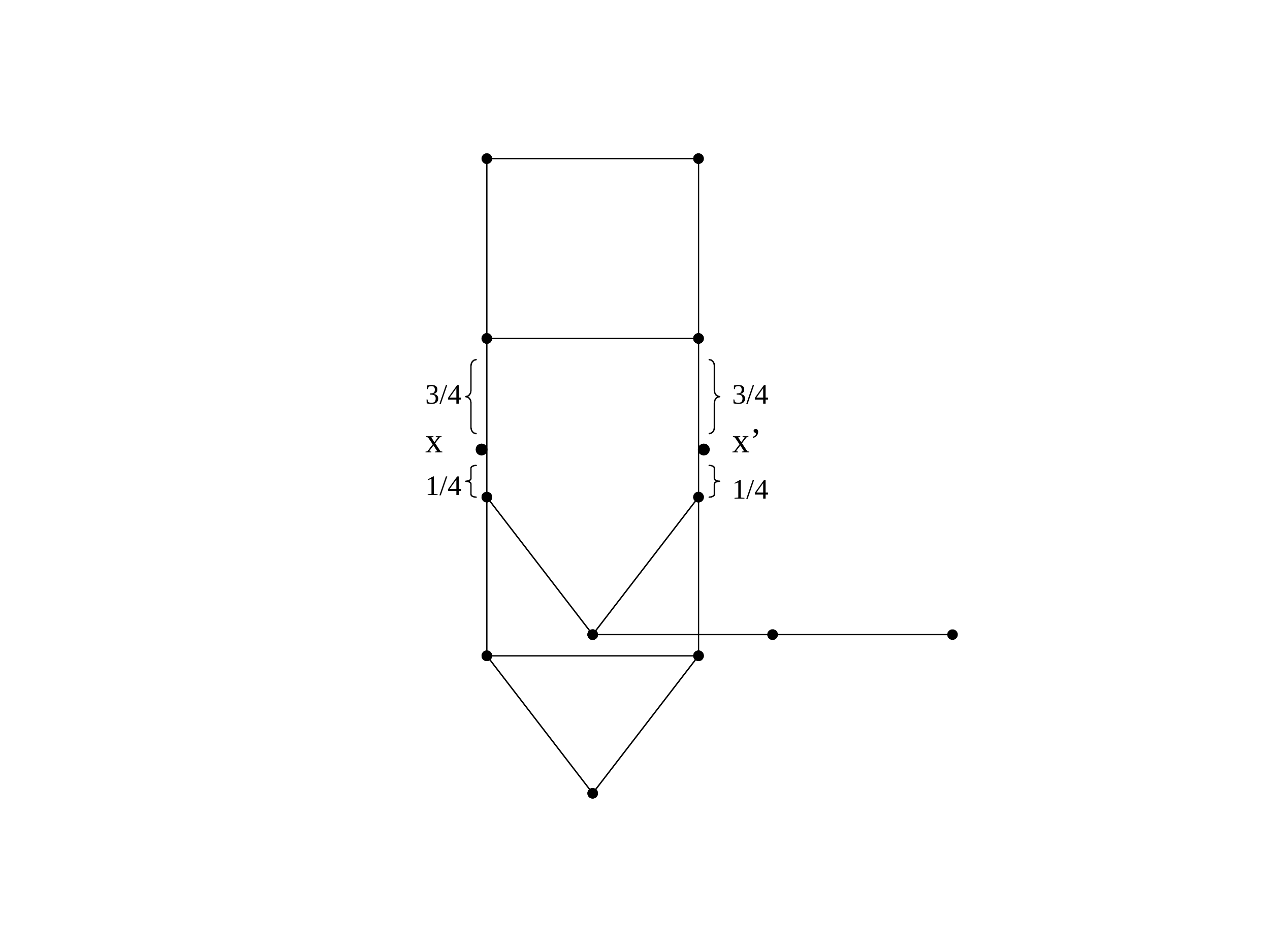}
\caption{The projection need not be open.}\label{ejpabierta}
\end{figure}

In this case, $p_\varepsilon$ is not open, not even restricted to $X\backslash p_\varepsilon^{-1}(X)$. If we consider some small enough
$\delta$ around $x$, for instance
$\delta\leq \frac{1}{8}$ it is immediate to see that
$p_\varepsilon^{-1}(p_\varepsilon (B(x,\delta)))=B(x,\delta)\cup
\{x'\}$, which is not open. This means that $p_\varepsilon
(B(x,\delta))$ is not open since, as we showed above,
$p_\varepsilon$ is continuous.

\begin{lema} \label{isometry 3} Let $\bar{B}(x,\varepsilon)\neq X$.
Then there is some $\delta>0$ such that for any connected
component $C_i\in B(x,\delta)\backslash \{x\}$ the restriction
$p_\varepsilon|:C_{i,\delta}\cup \{x\}\to p_\varepsilon(C_{i,\delta}\cup
\{x\})$ is an isometry.
\end{lema}

\begin{proof} If $\varepsilon \leq d_1$ it is immediate.

Else, let $y\in \partial \bar{B}(x,\varepsilon)\cap cl(X\backslash
\bar{B}(x,\varepsilon))$ and $\delta_1>0$ such that Lemma
\ref{dist unica}  holds for $x,y$. Any connected component
$C_{i,\delta_1}$ of $B(x,\delta_1)$ holds that $\forall z\in
C_{i,\delta_1} \quad d(z,y)=d(z,x)+d(x,y)$ or
$d(z,y)=d(x,y)-d(x,z)$. In the first case,
$\bar{B}(z,\varepsilon)\subset X\backslash B(y,d(x,z))$ and for any
$z'\in (z,x)$, since there is a path of length $\varepsilon +d(z',y)$
from $z'$ to $y$, $\bar{B}(z',\varepsilon)$ contains a point $p$ such
that $d(p,y)=d(z',x)$ and hence
$d_H(\bar{B}(z,\varepsilon),\bar{B}(z',\varepsilon))\geq d(z,z')$.

Now, suppose that $\forall z\in C_{i,\delta_1} \quad d(z,y)=d(x,y)-d(x,z)$. Applying Lemma \ref{dist unica} on $y,x$ we obtain a ball
about $y$, $B(y,\delta_2)$ so that, since $y\in cl(X\backslash
\bar{B}(x,\varepsilon))$, at least one of the components
$D_{1,\delta_2}$ will be such that $d(z',x)=d(z',y)+d(y,x) \ \forall \,
z'\in D_{1,\delta_2}$. Let $\delta<min\{\delta_1,\delta_2\}$.
$\forall z\in C_{i,\delta} \quad d(z,y)=d(x,y)-d(x,z)$ and hence
$\bar{B}(z,\varepsilon)\cap D_{1,\delta}= \bar{B}(y,d(x,z))\cap
D_{1,\delta}$. Thus $\forall z,z'\in C_{i,\delta}$
$d_H(\bar{B}(z,\varepsilon),\bar{B}(z',\varepsilon))\geq d(z,z')$.

Hence, by the properties of the length metric, for any pair of points
$z,z'\in X$ \ $d_H(\bar{B}(z,\varepsilon),\bar{B}(z',\varepsilon))\leq
d(z,z')$ finishing the proof.
\end{proof}

\begin{prop}\label{preim. finitos puntos} Let $X$ be a finite metric graph and let $x\in X$
and $\varepsilon>0$ be such that $\bar{B}(x,\varepsilon)\neq X$. Then,
$p_\varepsilon^{-1}(\bar{B}(x,\varepsilon))$ has a finite number of
points.
\end{prop}

\begin{proof} If $\bar{B}(x,\varepsilon)\neq X$, then there is a point
$y\in \partial \bar{B}(x,\varepsilon)\cap cl(X\backslash
\bar{B}(x,\varepsilon))$. For any point $x'$ such that
$\bar{B}(x',\varepsilon)=\bar{B}(x,\varepsilon)$ also $y\in \partial
\bar{B}(x',\varepsilon)\cap cl(X\backslash \bar{B}(x',\varepsilon))$ and
$d(x',y)=\varepsilon$. This means that
$p_\varepsilon^{-1}(\bar{B}(x,\varepsilon))$ is contained in
$S(y,\varepsilon)$ which is a finite number of points by Lemma \ref{finite
points}.
\end{proof}

\begin{lema}\label{dist minima} Fixed any $\varepsilon>0$,
$\forall x \in X$ and $\forall \delta_0>0$ there exists some
$\delta_1>0$ such that $\forall z \not \in
B(p_\varepsilon^{-1}(\bar{B}(x,\varepsilon)),\delta_0)$,
$d_H(\bar{B}(z,\varepsilon),\bar{B}(x,\varepsilon))>\delta_1$.
\end{lema}

\begin{proof} Otherwise, there would be some $\delta_0>0$ such that for each
$\delta_n=\frac{1}{n}$ there is a point $z_n\not \in
B(p_\varepsilon^{-1}(\bar{B}(x,\varepsilon)),\delta_0)$ for which
$d_H(\bar{B}(z_n,\varepsilon),\bar{B}(x,\varepsilon))\leq \frac{1}{n}$.

Since X is compact there is a cluster point $z$ of $(z_n)$ and
obviously $\bar{B}(z,\varepsilon)=\bar{B}(x,\varepsilon)$. Thus $z\in
p_\varepsilon^{-1}(\bar{B}(x,\varepsilon)$ and it is a cluster point of
$z_n$ which leads to a contradiction since $z_n\not \in
B(p_\varepsilon^{-1}(\bar{B}(x,\varepsilon)),\delta_0)$.
\end{proof}

\begin{prop} \label{preim. loc. conexa} If $X$ is a finite metric
graph and $\varepsilon>0$ is such that $X\in p_\varepsilon(X)$, then
$p_\varepsilon^{-1}(X)$ is locally connected.
\end{prop}

\begin{proof} Let $x\in p_\varepsilon^{-1}(X)$.
If $\partial \bar{B}(x,\varepsilon)=\emptyset$ then, since X is
compact, let $\varepsilon':=max_{y\in X}\{d(x,y)\}<\varepsilon$. Hence, if
$0<\delta<\varepsilon-\varepsilon'$, then $\forall z\in B(x,\delta) \quad
\bar{B}(z,\varepsilon)=X$.

If $\partial \bar{B}(x,\varepsilon)=\{y_1,\cdots,y_k\}$ we use Lemma
\ref{dist unica} with each $y_i$ and we consider a minimum
$\delta>0$ for the $k$ points such that each connected component
$C_i$ of $B(x,\delta)\backslash \{x\}$ is contained in some edge and
$\forall z\in C_i$ and $\forall i$, $d(z,y_i)=d(x,y_i) + d(x,z)$
or $d(z,y_i)=d(x,y_i) - d(x,z)$. Consider $C_1,\cdots,C_r$
those components for which $d(z,y_i)=d(x,y_i) - d(x,z)$ for every $i$ (if any).
If we also impose that $0<\delta< \varepsilon - \max_{a\in K_{x,\varepsilon}}\{d(a,x)\}$,
%
%
then $\bar{B}(z,\varepsilon)=X \ \forall
z\in C_i$ with $i\leq r$.

If $z\in C_i$ with $i>r$, then at least for one of the border
points $y_j$, $d(z,y_j)>\varepsilon$ and hence
$\bar{B}(z,\varepsilon)\neq X$.

Thus, $B(x,\delta)\cap p_\varepsilon^{-1}(X)=\{x\}\cup C_1\cup \cdots
\cup C_r$ and it is connected.
\end{proof}

\begin{prop}\label{finitas c c} If $X$ is a finite metric graph then for any $x\in X$
and any $\varepsilon>0$, $p_\varepsilon^{-1}(\bar{B}(x,\varepsilon))$ has a
finite number of connected components.
\end{prop}

\begin{proof} If $\bar{B}(x,\varepsilon)\neq X$, then by Proposition
\ref{preim. finitos puntos}, $p_\varepsilon^{-1}(\bar{B}(x,\varepsilon))$
has a finite number of points and we are done.

If $\bar{B}(x,\varepsilon)= X$, by Proposition \ref{preim. loc.
conexa}, $p_\varepsilon^{-1}(X)$ is locally connected which implies
that the connected components are open. Also, since $p_\varepsilon$ is continuous, $p_\varepsilon^{-1}(X)$  is
compact. Therefore, the connected components must be finite since they
define an open covering of a compact set.
\end{proof}

We take the following definitions and the characterization of a
graph from chapter IX in \cite{N1}. $Bd(V)$ denotes the boundary of a set $V$, this is,
$Bd(V)=\bar{V}\cap \bar(V^c)$.

\begin{definicion} Let $(X,T)$ a topological space, and let $A\subset X$.
Let $\beta$ be a cardinal number. We say that $A$ is \emph{of
order less than or equal to $\beta$ in X}, written \[ord(A,X)\leq \beta,\]
provided that for each $U\in T$ such that $A\subset U$, there
exists $V\in T$ such that \[A\subset V\subset U \mbox{ and }
|Bd(V)|\leq \beta.\]

We say that $A$ is \emph{of order $\beta$ in $X$}, written
\[ord(A,X)=\beta,\] provided that $ord(A,X)\leq \beta$ and $ord(A,X)\not \leq
\alpha$ for any cardinal number $\alpha<\beta$. If $A=\{p\}$ it is
usually denoted $ord(p,X)$ instead of $ord(\{p\},X)$.
\end{definicion}

\begin{teorema}\cite{N1}\label{caract. grafo} A \emph{continuum} $X$ (i.e. compact connected and metrizable) is a graph if and only if (1) and
(2) below both hold:

(1) $ord(x,X)<\aleph_0$ for all $x\in X$;

(2) $ord(x,X)\leq 2$ for all but finitely many $x\in X$.
\end{teorema}

\begin{prop} \label{orden finito} If X is a finite metric graph, then $\forall p\in
p_\varepsilon(X) \ ord(p,p_\varepsilon(X))<\aleph_0$.
\end{prop}

\begin{proof} Let $p\in p_\varepsilon(X)$. If $p\neq \{X\}$, by Proposition
\ref{preim. finitos puntos}, the inverse image is a finite number
of points. Let $p_\varepsilon^{-1}(p)=\{x_1,...,x_n\}$. Let $\delta>0$
such that Lemma \ref{isometry 3} holds for $x_1,...,x_n$. Then, by Lemma \ref{dist minima} there exists some
$\delta_1>0$ such that $\forall z\not \in
B(p_\varepsilon^{-1}(p),\delta)$ then
$d_H(\bar{B}(z,\varepsilon),p)>\delta_1$. (Assume also
$\delta<\frac{1}{2}d(x_i,x_i) \forall i\neq j$). So let
$\delta_0<\delta_1,\delta$ and let us study the boundary $Bd(P(p,\delta_0))$
where $P(p,\delta_0)$ represents the ball about $p$ of radius $\delta_0$ in
$p_\varepsilon(X)$ with the Hausdorff metric restricted from $2^X$.

First, note that Lemma \ref{dist minima} means that
$P(p,\delta_0)\subset p_\varepsilon(B(p_\varepsilon^{-1}(p),\delta))$.
But there is a unique point at distance $\delta_0<\delta$ from
$x_i$ in each connected component of $B(x_i,\delta)$ and then, by
Lemma \ref{isometry 3}, since there is a finite number of points
in $p_\varepsilon^{-1}(p)$ and the graph is locally finite, there will be a finite
number of different balls $p_j\in p_\varepsilon(X)$ such that
$d_H(p_j,p)=\delta_0$. Thus $|Bd(P(p,\delta_0))|<\aleph_0 \ \forall \,
\delta_0<\delta_1$ and $ord(p,p_\varepsilon(X))<\aleph_0$.

Otherwise, suppose $p=\{X\}$. By Lemma \ref{finitas c c}
$p_\varepsilon^{-1}(X)$ has finitely many connected components. Then
$Bd(p_\varepsilon^{-1}(X))$ has a finite number of points
$\{x_1,...,x_n\}$ which hold that $\partial \bar{B}(x_i,\varepsilon)\neq
\emptyset \ \forall i=1,n$.

Let $\delta>0$ such that Corolary \ref{isometry 2} holds for
$x_1,...,x_n$ and apply Lemma \ref{dist minima} to get some
$\delta_1>0$ such that $\forall z\not \in
B(p_\varepsilon^{-1}(X),\delta)$ then
$d_H(\bar{B}(z,\varepsilon),p)>\delta_1$. Now let
$\delta_0<\delta_1,\delta$ and let us study the boundary $Bd(P(p,\delta_0))$.
By Corolary \ref{isometry 2}, there is at most a finite number of points
$z_j$ in each ball $B(x_i,\delta)$ (assume also
$\delta<\frac{1}{2}d(x_i,x_i) \ \forall \, i\neq j$) such that
$d_H(\bar{B}(x_i,\varepsilon),\bar{B}(z_j,\varepsilon))=\delta_0$ and, by
Lemma \ref{dist minima}, any point whose ball is in $P(X,\delta)$
must be in one of those $B(x_i,\delta)$. Hence
$|Bd(P(X,\delta_0))|<\aleph_0 \ \forall \, \delta_0<\delta_1$ and
$ord(X,p_\varepsilon(X))<\aleph_0$.
\end{proof}

Using the inductive definition of dimension it is now immediate
the following.

\begin{cor} \label{1-dim} If $X$ is a finite metric graph then $p_\varepsilon(X)$
is 1-dimensional for every $\varepsilon$.
\end{cor}

\begin{proof} $\forall p\in p_\varepsilon(X) \
ord(p,p_\varepsilon(X))<\aleph_0$ which means that for any point $p$
there are arbitrarily small neighborhoods whose boundary consists
of finitely many points and these are obviously isolated.
\end{proof}

We are going to use the following characterization of being a graph in terms of the order from \cite[9.10]{N1} stated in Theorem \ref{caract. grafo}. Hence, by \ref{1-dim} we only need to prove that for every
point in $X$ but finitely many, the order is $\leq 2$, this is,
that there are neighborhoods in $p_\varepsilon(X)$ arbitrarily close
to the projection of that point whose boundary consists exactly on
two points.

\begin{teorema}\label{grafo} If $X$ is a finite metric graph, then for any $\varepsilon>0$
$p_\varepsilon(X)$ is a graph.
\end{teorema}

\begin{proof} First note that
$p_\varepsilon(X)$ is a continuum since $p_\varepsilon$ is continuous.

By Proposition \ref{orden finito}, we know that $\forall p\in p_\varepsilon(X)
\ ord(p,p_\varepsilon(X))<\aleph_0$.

By Proposition \ref{finite points proj}, $p_\varepsilon(X)\backslash
p_\varepsilon(A_\varepsilon)$ consists of a finite number of points, so
it suffices to check that
$ord(\bar{B}(x,\varepsilon),p_\varepsilon(X))=2 \ \forall \, x\in
A_\varepsilon$.

Let $x\in A_\varepsilon$. If there are not $x'\neq x$ such that
$\bar{B}(x,\varepsilon)=\bar{B}(x',\varepsilon)$ or, equivalently, if $p_\varepsilon^{-1}(p_\varepsilon(x))=\{x\}$, then there exists some
$\delta>0$ such that $\forall z\in B(x,\delta)$,
$p_\varepsilon^{-1}(p_\varepsilon(z))=\{z\}$. By Lemma
\ref{dist minima} we know that there exists some $\delta_0$ such that
$P(p,\delta_0)\subset
p_\varepsilon(B(p_\varepsilon^{-1}(p),\delta))$ (where
$p=p_\varepsilon(x)$) and therefore, it is immediate to see that the order is 2.

Let $x\in A_\varepsilon$ and $x'\neq x$ such that
$\bar{B}(x,\varepsilon)=\bar{B}(x',\varepsilon)$. Let $x\in [v_0,v_1]$,
$d_0=d(x,v_0),d_1=d(x,v_1)$ and $\{y_1,\cdots, y_n\}=\partial
\bar{B}(x,\varepsilon)\cap cl(X\backslash \bar{B}(x,\varepsilon))$. It
is obviously necessary that $\varepsilon>d_1$ to allow the existence
of such an $x'$. Since $\bar{B}(x,\varepsilon)\neq X$, $n\geq 1$
and there is some $y\in [w,w']$ such that $d(x,y)=d_0+k+d'_0=d_1+k'+d'_1$
and hence any geodesic segment $[x,y]$ contains a vertex $w$
with $d(w,y)=d'_0$ or $d(w,y)=d'_1$. Since $y$ is also a point in
$\bar{B}(x',\varepsilon)\cap cl(X\backslash \bar{B}(x',\varepsilon))$,
any geodesic segment $[y,x']$ has length $\varepsilon$ and
contains $[w,y]$. Then, $x'\in e'=[v'_0,v'_1]$ some edge with
$d(x',v'_0)=d_0$ and $d(x',v'_1)=d_1$.

Let $\gamma_i$ be a geodesic segment (of length $\varepsilon$) from $x$ to
$y_i$ and $\gamma'_i$ a geodesic segment (of length also $\varepsilon$)
from $x'$ to $y_i$. As we saw in Lemma \ref{dist unica 2}, there is a
partition of $\{y_1,...,y_n\}$ so that $y_i\in
[w_i,w'_i]$ with $d(w_i,x)=\varepsilon-d(w_i,y_i)$ and
$d(w_i,y_i)=d'_0$ if $i\leq k$ and $d(w_i,y_i)=d'_1$ if $i> k$.
Since $d_0'\neq d_1'$, by Lemma \ref{ditta dist borde}, $\gamma'_i$
contains $v'_0$ for $i\leq k$ and $\gamma'_i$ contains $v'_1$ for
$i> k$.

If we apply now Lemma \ref{dist unica 2} to $x'$ we will
immediately see that the election of the subsets
$\{y_1,...,y_k\}$ and $\{y_{k+1},...,y_n\}$ from Lemma \ref{subset border} is independent
from the center of the ball we are considering and, assuming $\delta$ small enough so that Lemma
\ref{balls 2} holds for both $x$
and $x'$, then $\forall z\in B(x,\delta),z'\in B(x',\delta)$ such that
$d(z,v_0)=d(z',v'_0)$, $B(z,\varepsilon)=B(z',\varepsilon)$. Note that
the description of the ball in this lemma only depends on the
initial ball, which is the same, on $\delta$, and on the partition of
the border points which also coincides.

By \ref{preim. finitos puntos}, $p_\varepsilon^{-1}(p_\varepsilon(x))$
consists of a finite number of points $\{x_1,...,x_n\}$ and, as
we have just seen, $x_i\in [v_0^i,v_1^i]$ for some edge with
$d(v_0^i,x_i)=d_0$ and $d(x_i,v_1^i)=d_1$. Now let $\delta_0>0$ be 
small enough so that  Lemma \ref{balls 2} holds $\forall x_i$ and, by
Lemma \ref{dist minima}, let $\delta_1$ be such that $\forall \, z \not
\in B(p_\varepsilon^{-1}(\bar{B}(x,\varepsilon)),\delta_0)$,
$d_H(\bar{B}(z,\varepsilon),\bar{B}(x,\varepsilon))>\delta_1$.

Consider any $\delta<\delta_0,\delta_1$ small enough so that Lemma \ref{isometry} also holds.
Assume also $\delta<\frac{1}{2}d(x_i,x_j)
\ \forall i\neq j$. Then, any point $z$ such that
$d_H(\bar{B}(z,\varepsilon),\bar{B}(x,\varepsilon))=\delta$ must be contained in
$B(p_\varepsilon^{-1}(\bar{B}(x,\varepsilon)),\delta_0)$ and, by Lemma
\ref{isometry}, $P(p_\varepsilon(x),\delta_0)$ is isometric to
$B(x_i,\delta_0)$ for any $x_i\in
p_\varepsilon^{-1}(\bar{B}(x,\varepsilon))$.  Thus $d(z,x_i)=\delta$ for
some $x_i\in p_\varepsilon^{-1}(\bar{B}(x,\varepsilon))$. But this gives
us two possible balls: $\bar{B}(z_0,\varepsilon)$ with $z_0\in
[v_0,x]$ and $d(z_0,x)=\delta$ and $\bar{B}(z_1,\varepsilon)$ with
$z_1\in [v_1,x]$ and $d(z_1,x)=\delta$. Any other ball coincides
with one of those since we saw before that for any other point
$z_0^j\in [v^j_0,x_j]$ with $d(z^j_0,x_j)=\delta$ or $z_1^j\in
[v^j_1,x_j]$ with $d(z^j_1,x_j)=\delta$ then
$\bar{B}(z^j_0,\varepsilon)=\bar{B}(z_0,\varepsilon)$ and
$\bar{B}(z^j_1,\varepsilon)=\bar{B}(z_1,\varepsilon)$.

Thus, $\forall x\in A_\varepsilon \
ord(\bar{B}(x,\varepsilon),p_\varepsilon(X))=2$ and $p_\varepsilon(X)$ is a
graph.
\end{proof}

\begin{quote} Now, to conclude this analisis of the semiflow for finite graphs we prove that through the different levels, the projection takes on a finite number of topological types. Therefore, every topological property is a Weierstrass type property. We also prove that the Euler characteristic is bounded for the evolution giving a lower bound for the projection (for which, as we show in Example \ref{horm}, may be smaller than the initial). This bound depends only on the number of edges of the original graph.
\end{quote}

For any finite graph $X$ and any $0<\varepsilon\leq diam(X)$
let us define a new graph $X_\varepsilon$ as a subdivision of $X$ as follows.
Let $\varepsilon=k\cdot \frac{1}{2}+\varepsilon_0$ with $0\leq
\varepsilon_0<\frac{1}{2}$ and $k\in \mathbb{N}$. If $\varepsilon_0=0$ we divide each edge in
two and the middle points of the edges become vertices of
$X_\varepsilon$. If $\varepsilon_0=\frac{1}{4}$ we divide each edge in
four, each of them with length $\frac{1}{4}$ adding three new
vertices. If $0<\varepsilon_0<\frac{1}{4}$ we define 5 new vertices
$w_1,\cdots w_5$ on each edge $[v,v']$ such that $w_3$ is the
middle point of the edge, $[v,w_1],[w_2,w_3],[w_3,w_4]$ and
$[w_5,v']$ have length $\varepsilon_0$ and $[w_1,w_2],[w_4,w_5]$ have
length $\frac{1}{2}-2\varepsilon_0$. Finally, if
$\frac{1}{4}<\varepsilon_0$ we define 5 new vertices $w_1,\cdots w_5$
on each edge $[v,v']$ such that $w_3$ is the middle point of the
edge, $[v,w_1],[w_2,w_3],[w_3,w_4]$ and $[w_5,v']$ have length
$\frac{1}{2}-\varepsilon_0$ and $[w_1,w_2],[w_4,w_5]$ have length
$2\varepsilon_0-\frac{1}{2}$. In both cases we divide each edge
$[v,v']$ in six parts. We obtain a new graph and a canonical
isometry $i:X_\varepsilon \to X$.

\begin{definicion} Let $X_\varepsilon / \sim_{\varepsilon}$ the quotient space
under the relation $x \sim_\varepsilon x'$ if and only if
$\bar{B}(x,\varepsilon)=\bar{B}(x',\varepsilon)$.
\end{definicion}

Given two topological spaces $A$ and $B$, $A\cong B$ will denote that $A$ and $B$ are homeomorphic.

\begin{nota} Obviously, $X_\varepsilon / \sim_{\varepsilon} \cong
p_\varepsilon(X)$.
\end{nota}

\begin{teorema}\label{identif} $X_\varepsilon / \sim{_\varepsilon}$ is a graph where the relation $\sim_\varepsilon$ holds that: \begin{itemize} \item[(a)] A point $x$ in the interior of an edge of $X_\varepsilon$ such that $\bar{B}(x,\varepsilon)\neq X$ is related to another point if and only if both are in different edges and those edges are identified in the quotient.
\item[(b)] A vertex $v$ of $X_\varepsilon$ such that $\bar{B}(v,\varepsilon)\neq X$ can only be related with another vertex.
\item[(c)] If $x$ is an interior point of an edge of $X_\varepsilon$ such that $\bar{B}(x,\varepsilon)= X$, then  $\bar{B}(y,\varepsilon)= X$ for every point $y$ in that edge. In this case, that edge defines a vertex in $X_\varepsilon /
\sim_{\varepsilon} \cong p_\varepsilon(X)$.\end{itemize}
\end{teorema}

\begin{proof} If $\varepsilon\leq 1$ it's readily seen that for any pair of points their balls only coincide if they contain the whole space and therefore the theorem holds.

Suppose $\varepsilon>1$ and let
$\varepsilon=k\frac{1}{2}+\varepsilon_0$ with $0\leq
\varepsilon_0<\frac{1}{2}$. Let $m$ the middle point of an edge and
$v$ one of its vertices. Any point in the interior of the half edge $z\in (v,m)$ immediately belogs to
$A_\varepsilon$ if $d(z,v)\neq \varepsilon_0$ and
$d(z,m)\neq\varepsilon_0$. Consider the graph $X_\varepsilon$
homeomorphic to $X$ as it is described above.

If $\varepsilon_0=0$, then $(v,m)\subset A_\varepsilon$. If
$\varepsilon_0=\frac{1}{4}$ and $x_1$ represents the middle point of
$(v,m)$, then $A_1=(v,x_1)\subset A_\varepsilon$ and
$A_2=(x_1,m)\subset A_\varepsilon$. If $\varepsilon_0\neq
0,\frac{1}{4}$, then there are two points $x_1,x_2\in (v,m)$ such that, at least $A_1=(v,x_1)\subset A_\varepsilon$,
$A_2=(x_1,x_2)\subset A_\varepsilon$ and $A_3=(x_2,m)\subset
A_\varepsilon$.

To prove $a)$, let $A_i\subset A_\varepsilon$ with $i\leq 3$ in any of the cases. Let $z\in A_i$ and $z'\in X$ such that
$\bar{B}(z,\varepsilon)=\bar{B}(z',\varepsilon)$. As we saw in the proof of Theorem \ref{grafo}, $z'$ is contained in some edge $[v_1,v_2]$ with $d(z',v_1)=d(z,v)$,
and if we consider $[v_1,m']$ whith $m'$
 the middle point, then $z'\in A'_i\subset
A_\varepsilon$. Also, as we saw in the same proof, there exists some $\delta>0$ such that for any $y\in
B(z,\delta)$, $y'\in B(z',\delta)$, if $d(y,v)=d(y',v_1)$ the closed balls
coincide,
$\bar{B}(y,\varepsilon)=\bar{B}(y',\varepsilon)$. Thus, the points of $A_i$ for which there exists a point in $A'_i$ whose image by
$p_\varepsilon$ is the same, form an open set in $A_i$. On the other side, if $y_n\in A_i$ is a sequence convergent to $y$, $y'_n\in
A'_i$ is a sequence convergent to $y'$ and
$\bar{B}(y_n,\varepsilon)=\bar{B}(y'_n,\varepsilon)$ for every
$n$, obviously $\bar{B}(y,\varepsilon)=\bar{B}(y',\varepsilon)$ and therefore, the points in $A_i$ for which there exist a point in
$A'_i$ whose image by $p_\varepsilon$ coincides is  open and closed in $A_i$ with $A_i$ conected. This implies that if there exists such a pair of points $z,z'$ then $\forall y\in A_i$ y $\forall
y'\in A'_i$ with $d(y,v)=d(y',v_1)$,
$p_\varepsilon(y)=p_\varepsilon(y')$.

$(b)$ is an immediate consequence from $(a)$.

Let us check $(c)$. Let $A_i\subset A_\varepsilon$ with
$i\leq 3$ in any of the cases and suppose $z\in A_i$ such that
$\bar{B}(z,\varepsilon)=X$. As we saw in the proof of Lemma \ref{balls 2}, if $z$ is contained in any cycle of length $\leq 2\varepsilon$ ($A_i\subset e$ is part of it) then the whole cycle is contained in the closed ball about any point of $e$ and, in particular, about any point of $A_i$. Any other border point of $\bar{B}(z,\varepsilon)=X$ must be a middle point or a vertex but this is not possible for a point in $A_i$ and we must conclude that $\forall z'\in A_i \quad
\bar{B}(z',\varepsilon)=X$.
\end{proof}

\begin{prop} $\forall \varepsilon>0$ the Euler's characteristic of
$p_\varepsilon(X)$ holds that $\aleph (p_\varepsilon(X))\geq
1-6|\mathcal{E}|$ where $|\mathcal{E}|$ denotes the number of
edges in $X$.
\end{prop}

\begin{proof} Let $\varepsilon\leq 1$. If $X$ has at least two edges, then
every pair of points have different balls. If $X$ consists
just of one edge, then $p_\varepsilon(X)$ is also contractible and the
proposition holds.

Suppose $\varepsilon>1$. Clearly, $X$ and
$X_\varepsilon$ have the same Euler's characteristic. Let us check what happens with the possible identifications of $[v,m]$ in
$X/\sim$ as in Lemma \ref{quotient}. Since we have at most
$6|\mathcal{E}|$ edges and at least 1 vertex, $\aleph(X/\sim)\geq
1-6|\mathcal{E}|$. In fact, if $N_0(\varepsilon)$ is the number of subsets
$A_i(\varepsilon)$ such that $\forall z'\in A_i \quad
\bar{B}(z',\varepsilon)=X$ then $\aleph(X/\sim)\geq 1+2 N_0(\varepsilon) -
6|\mathcal{E}|$.
\end{proof}

\begin{cor} For every finite graph $X$, $H_1(p_\varepsilon(X))=\mathbb{Z}^m$ with $m\leq
6|\mathcal{E}|-N_0(\varepsilon)$.
\end{cor}

\begin{definicion} By a \emph{critical time} we mean any $\varepsilon$ for which there
is a sequence $0<t_1<t_2<\cdots <\varepsilon$ convergent to
$\varepsilon$ and such that for every $i>0$ $f_{t_i}(X)$ is not
homeomorphic to $p_{\varepsilon}(X)$.
\end{definicion}

\begin{lema}\label{ball growth} Let $C_1,C_2,... ,C_k$ be the minimal cycles in $X$ with lengths
$l(C_1)=l_1,\cdots , l(C_k)=l_k$. Let
$\varepsilon'>0$ such that $2\varepsilon' \neq l_i \ \forall \, i$ and $x\in A_{\varepsilon'}$. If
$\{y_1,\cdots,y_n\}=\partial \bar{B}(x,\varepsilon')$, then there
exists some $\delta_0>0$ such that if
$\varepsilon'-\varepsilon=\delta<\delta_0$ then
$\bar{B}(x,\varepsilon)=\bar{B}(x,\varepsilon')\backslash
\{\cup_{i=1}^n B(y_i,\delta)\}$.
\end{lema}

\begin{proof} Claim.  $\partial \bar{B}(x,\varepsilon')\subset cl(X\backslash \bar{B}(x,\varepsilon'))$.

Suppose $y\in \partial \bar{B}(x,\varepsilon')\backslash cl(X\backslash \bar{B}(x,\varepsilon'))$ with $x\in [v_0,v_1]$ and $y\in[u_0,u_1]$. Since $y\in \bar{B}(x,\varepsilon')\backslash cl(X\backslash \bar{B}(x,\varepsilon'))$, $[u_0,u_1]\subset \bar{B}(x,\varepsilon')$ and, in particular, there are two geodesic segments of length $\varepsilon'$, $\gamma_0$, $\gamma_1$, such that $u_0\in \gamma_0$ and $u_1\in \gamma_1$.

If there is $v_i$ with $i=0,1$ so that $v_i\in \gamma_0\cap \gamma_1$ then $d(x,v_i)+d(v_i,u_0)+d(u_0,y)=\varepsilon'=d(x,v_i)+d(v_i,u_1)+d(u_1,y)$. Hence, $d(v_i,u_0)+d(u_0,y)=d(v_i,u_1)+d(u_1,y)$ and, since $d(v_i,u_j)\in \mathbb{Z}$ for $j=0,1$, then $|d(u_0,y)-d(u_1,y)|$ is an integer number which means that $y$ is either a vertex or a middle point. This contradicts the fact that $x\in A_\varepsilon$ and, therefore, $d(y,u_i)\neq 0,\frac{1}{2},1$.

Thus, we may assume, relabelling $u_0,u_1$ if necessary, that $v_0,u_0\in \gamma_0$ and $v_1,u_1\in \gamma_1$. As we saw in  Lemma \ref{Lema: ciclos}, see Figure \ref{fig_grafo}, this implies that $\gamma_0,\gamma_1$ form a minimal cycle of length $2\varepsilon'$ which is a contradiction. This proves the claim.

Let $\{y_1,\cdots,y_n\}=\partial \bar{B}(x,\varepsilon')\subset
cl(X\backslash \bar{B}(x,\varepsilon'))$. Let $\delta_0<\delta_2$ be such that
the balls $\bar{B}(y_i,\delta_0)$ are disjoint and contained is some edge. Suppose also that $\varepsilon'-\delta_0>\max_{a\in K_{x,\varepsilon'}}\{d(a,x)\}$ (in particular, $\frac{l_i}{2}\not \in (\varepsilon'-\delta_0,\varepsilon')$). It is immediate to check that if 
$\varepsilon=\varepsilon'-\delta$ for some $\delta<\delta_0$
$\bar{B}(x,\varepsilon)=\bar{B}(x,\varepsilon')\backslash
\{\cup_{i=1}^n B(y_i,\delta)\}$.
\end{proof}

The main point of this lemma is that for any $\varepsilon'$ but a finite number and any $x\in A_{\varepsilon'}$ the ball about $x$ of radius $\varepsilon=\varepsilon'-\delta$ (with $\delta$ small enough) is determined by the ball of radius $\varepsilon'$ and the number $\delta$ independently of its center $x$.

\begin{obs}\label{Obs: simplicial} Let $X$ be a finite metric graph and let $\varepsilon' >0$ such that $\varepsilon'=k\cdot \frac{1}{2}+\varepsilon'_0$ with $0<\varepsilon'_0<\frac{1}{2}$ and $\varepsilon'_0\neq\frac{1}{4}$. Suppose $\delta_0<\varepsilon'_0,|\varepsilon'_0-\frac{1}{4}|$. Then, for any $\varepsilon\in (\varepsilon'-\delta_0, \varepsilon')$ there is a canonical simplicial map  $i_{\varepsilon,\varepsilon'}\colon X_\varepsilon \to X_{\varepsilon'}$ which is an isomorphism.
\end{obs}

\begin{teorema} For every finite metric graph $X$ there is a finite number of critical
times.
\end{teorema}

\begin{proof} Let $C_1,... ,C_k$ be the minimal cycles in $X$ with lengths
$l(C_1)=l_1,... , l(C_k)=l_k$. Let $\varepsilon'>0$ be such that $2\varepsilon'\neq l_i \ \forall i$ and suppose
$\varepsilon'=k\cdot \frac{1}{2}+ \varepsilon'_0$ with
$0< \varepsilon'_0<\frac{1}{2}$ and $\varepsilon'_0\neq
0,\frac{1}{4}$. Note that this includes every possible  $\varepsilon'\leq diam(X)$ but a finite number.

Now let $0<\delta_0$ such that
$\delta_0<|2\varepsilon'-l_i| \ \forall i$ and
$\delta_0<\varepsilon'_0,|\varepsilon'_0-\frac{1}{4}|$. Then we claim
that there is some $\delta\leq \delta_0$ such that $\forall
\varepsilon \in (\varepsilon'-\delta,\varepsilon')$,
$p_\varepsilon(X)\cong p_{\varepsilon'}(X)$ and $\varepsilon'$ is not a
critical time. We are going to prove the existence of some $\delta$ for each edge and each vertex of $X_{\varepsilon'}$. Then, since they are a finite number, it suffices to take the minimun.

First, let us study the edges of $X_{\varepsilon'}$.

\underline{Case 1} Suppose $\varepsilon'_0<\frac{1}{4}$. Let $(v,m)$
half an edge with $v$ a vertex and $m$ a middle point. Consider the graph $X_{\varepsilon'}$ defined above and let
$w_1,w_2\in (v,m)$ be the points dividing $(v,m)$ in three
parts: $A_1^{\varepsilon'}=(v,w_1)$ of length $\varepsilon'_0$,
$A_2^{\varepsilon'}=(w_1,w_2)$ with length
$\frac{1}{2}-2\varepsilon'_0$ and $A_3^{\varepsilon'}=(w_3,m)$ with
length $\varepsilon'_0$.

Consider also the graph $X_{\varepsilon}$ with $\varepsilon'-\varepsilon<\delta_0$. Then, by Remark \ref{Obs: simplicial}, the half edge $(v,m)$ is divided in three parts $A_1^\varepsilon$, $A_2^\varepsilon$, $A_3^{\varepsilon'}$ and $i_{\varepsilon,\varepsilon'}\colon X_\varepsilon \to X_{\varepsilon'}$ holds that
$i_{\varepsilon,\varepsilon'}(A_i^\varepsilon)=A_i^{\varepsilon'}$ for $i=1,2,3$.

Since $\varepsilon'-\varepsilon<\delta_0$, there exists some $x_i\in A_i^{\varepsilon'}\cap A_i^{\varepsilon} \quad \forall \, i=1,2,3$. Clearly, either $x_i\in A_{\varepsilon'}$ or $\bar{B}(x_i,\varepsilon')=X$ with no vertices in the border.

If $x_i\in A_{\varepsilon'}$, then, by Lemma \ref{ball growth}, there is some $\delta>0$ such that if $\varepsilon'-\varepsilon<\delta$, then
$x_i\sim_{\varepsilon'} x'_i$ if and only if $x_i \sim_\varepsilon
x'_i$. 

If $\bar{B}(x_i,\varepsilon')=X$
since $2\varepsilon'\neq l_i$ $\forall \, i$, by Lemma \ref{Lema: ciclos}, $\partial B(x_i,\varepsilon')=\emptyset$.
Therefore, there exists $\delta>0$ such that $\forall \varepsilon\in (\varepsilon'-\delta,\varepsilon)$, $\bar{B}(x_i,\varepsilon)=X$. The same holds for $x'_i\in {A'}_i^{\varepsilon'}$ and, therefore, $x_i\sim_{\varepsilon'} x'_i$ if and only if $x_i \sim_\varepsilon x'_i$. By Theorem \ref{identif},  $A_i^{\varepsilon'}\sim_{\varepsilon'} {A'}_i^{\varepsilon'}$ if and only if $A_i^{\varepsilon}\sim_\varepsilon {A'}_i^{\varepsilon}$.

Hence, taking the minimum $\delta$ over all $A_i^{\varepsilon'}$, we conclude that the identification of edges
in $X_\varepsilon$ is the same as in $X_{\varepsilon'}$.

\underline{Case 2} If $\frac{1}{4}<\varepsilon'_0<\frac{1}{2}$ the
argument is analogous redefining the partitions $A_i^{\varepsilon}$ and $A_i^{\varepsilon'}$.

Consider now any vertex: $w_0=v,w_1,w_2$ or $w_3=m$ of $X_{\varepsilon'}$.

If $w_i\not\sim_{\varepsilon'} w'_i$ then $w_i\not\sim_{\varepsilon} w'_i$ for any $\varepsilon<\varepsilon'$ 
since, in a length space, once the balls about two points coincide for some radius they coincide also for any bigger radius.
Thus, it suffices to check that if $w_i\sim_{\varepsilon'}w_i'$, then there exists some
$0<\delta<\delta_0$ such that, 
$\forall
\varepsilon \in (\varepsilon'-\delta,\varepsilon')$ 
the corresponding vertices  in
$X_\varepsilon$, $v_i=i_{\varepsilon,\varepsilon'}^{-1}(w_i)$,$v'_i=i_{\varepsilon,\varepsilon'}^{-1}(w'_i)$ hold that $v_i\sim_\varepsilon
v'_i$.

First note that for the condition to fail, the edges adjacent to $w_i$ and $w'_i$ can't be identified at level
$\varepsilon'$. Otherwise, as we just saw, there would exist some $\delta$ so that for any  $\varepsilon \in [\varepsilon'-\delta,\varepsilon']$ the corresponding edges are identified and, with them, the vertices in their closure.

Also, if every point in the border of the ball were in the interior of an edge of $X$ different from the middle point, using the same argument from \ref{ball
growth} and assuming $\delta$ small enough, we obtain that $w_i\sim_\varepsilon w'_i$ with $w_i, w'_i$ contained in edges adjacent to $v_i$ and $v'_i$. Therefore, those edges are identified at level
$\varepsilon$ and, with them, the vertices $v_i$ and $v'_i$ in their closure. 

Thus, let us see the case where $\partial \bar{B}(w_i,\varepsilon')=\partial \bar{B}(w_i,\varepsilon')$ contains vertices or middle point of edges in $X$. Since we assumed that
$0<\varepsilon_0<\frac{1}{2}$, it suffices to consider the vertices in
$X_{\varepsilon'}$ which are not vertices nor middle points of edges in $X$.

So, let us suppose that $w,w$ are two vertices of $X_{\varepsilon'}$ which are not vertices nor middle point of edges in $X$, suppose that $w\sim_{\varepsilon'}w'$, suppose that no edge of $X_{\varepsilon'}$ adjacent to $w$ is identified with any other edge adjacent to $w'$. Suppose $\partial \bar{B}(w,\varepsilon')$ and $\partial \bar{B}(w,\varepsilon')$ contain vertices or middle point of edges in $X$ and suppose that for every $0<\delta<\delta_0$ there is some $\varepsilon \in (\varepsilon'-\delta,\varepsilon')$ such that $v=i_{\varepsilon,\varepsilon'}^{-1}(w)$,$v'=i_{\varepsilon,\varepsilon'}^{-1}(w')$  hold that $v\sim_{\varepsilon}v'$. This will lead to contradiction.

Let us fix $[v,m]$ the middle edge in $X$ containing $w$ and
$[v',m']$ the middle edge in $X$ containing $w'$. Let us relabell $\{a,b\}=\{v,m\}$ so that $d(w,a)=\varepsilon'_0$, $d(w,b)=\frac{1}{2}-\varepsilon'_0$ and $\{a',b'\}=\{v',m'\}$ so that $d(w',a')=\varepsilon'_0$, $d(w',b')=\frac{1}{2}-\varepsilon'_0$. See figure \ref{figura_5_39}.
\begin{quote} This takes account of all the possible cases: 
\begin{itemize}
	\item[a)] If $0<\varepsilon_0<\frac{1}{4}$ and $d(w,v)=d(w',v')=\varepsilon_0$.
	\item[b)] If $0<\varepsilon_0<\frac{1}{4}$ and $d(w,m)=d(w',m')=\varepsilon_0$.
	\item[c)] If $\frac{1}{4}<\varepsilon_0<\frac{1}{2}$ and $d(w,v)=d(w',v')=\frac{1}{2}-\varepsilon_0$.
	\item[d)] If $\frac{1}{4}<\varepsilon_0<\frac{1}{2}$ and $d(w,v)=d(w',v')=\frac{1}{2}-\varepsilon_0$.
\end{itemize}
\end{quote}

\begin{figure}[ht]
\includegraphics[scale=0.4]{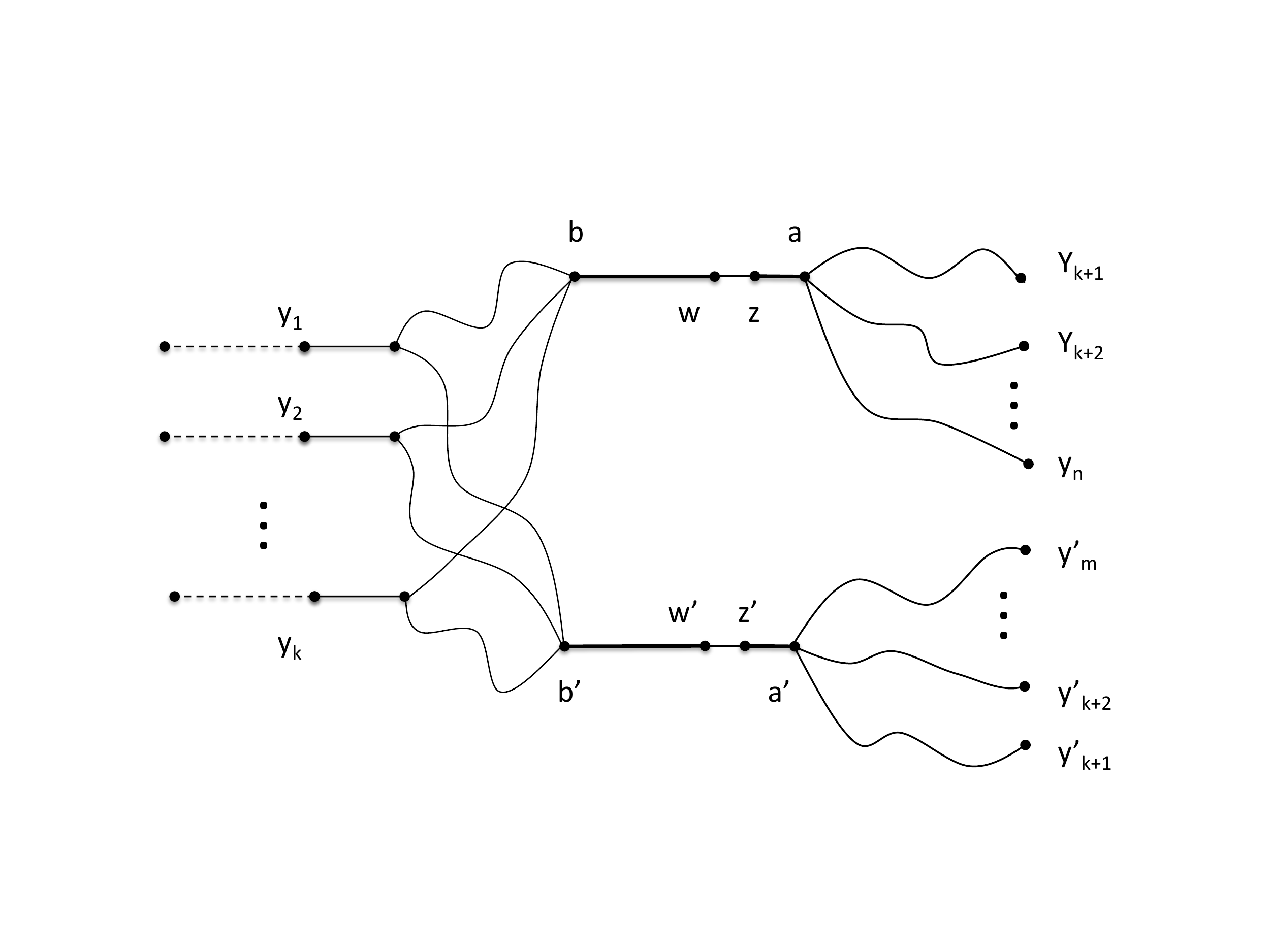}
\caption{Ideal representation where all the border points are in different edges.}\label{figura_5_39}
\end{figure}

We are considering $\varepsilon'_0\neq \frac{1}{4}$,  $|\varepsilon'_0-\frac{1}{2}|\neq k\frac{1}{2}$ for any $k\in \mathbb{Z}$. This means that there is a canonical partition of $\partial B(w,\varepsilon')=\{y_1,...,y_k\}\cup\{y_{k+1},...,y_n\}$ where $\{y_1,...,y_k\}$ are the border points which are not vertices nor middle points, i.e., border poins $y_j$ so that the geodesic segment $[w,y_j]$ contains $b$, and $\{y_{k+1},...,y_n\}$ are the border points which are vertices or middle points, i.e., border poins $y_j$ so that the geodesic segment $[w,y_j]$ contains $a$. 

Also, it is immediate to check that $\{y_1,...,y_k\}\subset cl(X\backslash \bar{B}(w,\varepsilon))$.

The same argument works for $w'$ and, since $cl(X\backslash \bar{B}(w,\varepsilon'))=cl(X\backslash \bar{B}(w',\varepsilon'))$ we obtain that $\partial B(w',\varepsilon')=\{y_1,...,y_k\}\cup\{y'_{k+1},...,y'_m\}$ where $\{y_1,...,y_k\}$ are the border points which are not vertices nor middle points, i.e., border poins so that the geodesic segment $[w',y_j]$ contains $b'$ (and these are the same for $w$ and $w'$), and $\{y'_{k+1},...,y'_m\}$ are the border points which are vertices or middle points, i.e., border poins $y'_j$ so that the geodesic segment $[w,y'_j]$ contains $a'$. (Notice that $\{y'_{k+1},...,y'_m\}$ need not be equal to $\{y_{k+1},...,y_n\}$ if there are border points which are not in $cl(X\backslash \bar{B}(w,\varepsilon'))$. See figure \ref{ejpabierta}).

Let $C$ be the set of edges in $X$ contained in the open ball $B(w,\varepsilon')$ and $C'$ be the set of edges in $X$ contained in the open ball $B(w',\varepsilon')$. Since the realizations $|C|, |C'|$ are compact, there is some $\delta_1$  such that $|C|\subset B(w,\varepsilon'-\delta_1)$ and $|C'|\subset B(w',\varepsilon'-\delta_1)$.

Let $\delta <\varepsilon_0,\delta_1$ be such that the balls $B(y_j,2\delta)$ are disjoint and contained in the interior of half an edge of $X$ $\forall \, 1\leq j\leq k$. Let $z\in [a,w]$ so that $d(z,w)=\delta$ and $z'\in [a',w']$ so that $d(z',w')=\delta$. 

Claim: $\bar{B}(z,\varepsilon'-\delta)=\bar{B}(w,\varepsilon')\backslash \Big(\cup_{j=1}^k B(y_j,2\delta) \Big)=\bar{B}(w',\varepsilon')\backslash \Big(\cup_{j=1}^k B(y_j,2\delta) \Big)=\bar{B}(z',\varepsilon'-\delta)$.

First, notice that $|C|\subset B(w,\varepsilon'-\delta)$ and $|C'|\subset B(w',\varepsilon'-\delta)$. By Lemma \ref{dist unica}, it is immediate to check that $d(z,y_j)=d(w,y_j)-d(w,z)=d(w,y_j)-\delta$ for every $k+1\leq j \leq n$ and $d(z',y'_j)=d(w',y'_j)-d(w',z')=d(w',y'_j)-\delta$ for every $k+1\leq j \leq m$. Hence, for every edge $e$ containing a point $y_j$, $k+1\leq j \leq n$, the geodesic segment $\gamma=[w,y_j]$  holds that $\gamma\cap e\subset \bar{B}(z,\varepsilon'-\delta)$. Also, for any edge $e'$ containing a point $y'_j$, $k+1\leq j \leq m$, the geodesic segment $\gamma'=[w',y'_j]$ holds that $\gamma'\cap e'\subset \bar{B}(z',\varepsilon'-\delta)$.
The only edges left are those containing the points $\{y_1,...,y_j\}$. Also, by Lemma \ref{dist unica}, we have that $d(z,y_j)=d(w,y_j)+d(w,z)=d(w,y_j)+\delta$  and $d(z',y_j)=d(w',y_j)+d(w',z')=d(w',y_j)+\delta$ for every $1\leq j \leq k$. Then, by the election of $\delta$, it is clear that for any edge $e$ contaning a point $y_j$, $1\leq j\leq k$, and for any geodesic segment $\gamma=[w,y_j]$, $(\gamma\cap e)\cap \bar{B}(z,\varepsilon'-\delta)=(\gamma\cap e)\backslash B(y_j,2\delta)$ and for any geodesic segment $\gamma'=[w',y_j]$, $(\gamma'\cap e)\cap \bar{B}(z',\varepsilon'-\delta)=(\gamma'\cap e)\backslash B(y_j,2\delta)$. 
Therefore, we conclude the claim.

Hence, $\bar{B}(z,\varepsilon'-\delta)=\bar{B}(z',\varepsilon'-\delta)$. In particular, $z\sim_{\varepsilon'} z'$ and we obtain the contradiction since there are no adjacent edges identified.
\end{proof}

\begin{cor} For every finite metric graph $X$  there is a finite number of possible
topological types on the set of projections $\{p_\varepsilon(X) \, : \, \varepsilon>0\}$.
\end{cor}

\section[Geometrical models  for the semiflow]{ $\mathbb{R}$-trees: geometrical models  for the semiflow}\label{section trees}

A \emph{real tree} or \emph{$\mathbb{R}$--tree} is a metric space $(T,d)$
that is uniquely arcwise connected and $\forall \, x, y \in T$ the
unique arc from $x$ to $y$, denoted $[x,y]$, is isometric to the
subinterval $[0,d(x,y)]$ of $\mathbb{R}$. In this section we are going to use $\mathbb{R}$-trees and their ends spaces, in the sense of \cite{Hug} (see also \cite{Mart-Mor} and \cite{Hug-M-M}), to produce a geometrical model for the semiflow and use it to describe the identification process of the closed balls in the levels of the semiflow.

We propose here the idea that some properties of the semiflow can be reflected into geometrical properties of end spaces of certain $\mathbb{R}$-trees. We treat herein the property of being topologically robust. To do this we use Whitney functions on hyperspaces. Let us recall the definition from \cite{N2}.

\begin{definicion} Let $\mathcal{H}=2^X$ or $C(X)$. A Whitney function in $\mathcal{H}$ is a continuous function $w:\mathcal{H}
\to [0,+\infty)$ satisfying:

\begin{itemize} \item[(a)] If $A,B\in \mathcal{H}$ are such that $A\subset
B$ and $A\neq B$ then $w(A)<w(B)$. \item[(b)] $w(\{x\})=0$ for every $x\in X$.
\end{itemize}
\end{definicion}

It is well known the existence of a Whitney map $w: 2^X_H \to
[0,\infty)$ for every nonempty compact metric space, see \cite{N2}. A natural way
of defining levels on the semiflow, instead of considering projections
$p_\varepsilon$, would be to consider Whitney levels restricted to the semiflow
$\mathcal{B}:=\{\bar{B}(x,\varepsilon \ | x\in X \mbox{ and }
\varepsilon>0)\}$: $w^{-1}(t)\cap \mathcal{B}$.

Nevertheless, the behavior of those levels doesn't work for some
of the results given here. For example, theorems \ref{flujo
homeom.geod} and \ref{flujo homeom.pol} would fail even for very simple
examples of Peano continua.

Let us define a projection $\pi_t:X \to w^{-1}(t)\cap \mathcal{B}$
sending each point $x\in X$ to the unique closed ball
$\bar{B}(x,\varepsilon)\in \mathcal{B}$ such that
$w(\bar{B}(x,\varepsilon))=t$. This is well defined since $w$ is
continuous and strictly increasing.

First, let's see that if we try to do the same using levels defined by a Whitney function we lose information, specially with respect to the map relating the initial space with the corresponding level.

\begin{ejp} Let $X$ be the $[0,1]$ subinterval of the real line with
the euclidean metric. For any Whitney function $w: 2^X_H \to
[0,\infty)$ and for every $\varepsilon_0>0$ there exists
$t<\varepsilon_0$ such that $\pi_t:X \to w^{-1}(t)\cap \mathcal{B}$
is not a homeomorphism.

Let $0<\varepsilon<1$ such that $w(\bar{B}(0,\varepsilon))=t$. Consider the closed ball
$\bar{B}(\frac{\varepsilon}{2},\frac{\varepsilon}{2})$. Clearly these
two balls coincide and thus $\pi_t(0)=\pi_t(\frac{\varepsilon}{2})$
and $\pi_t$ is not injective nor a homeomorphism.
\end{ejp}

Nevertheless we can establish some relation between Whitney levels and the levels defined in the semiflow.

\begin{lema} For every $0<\varepsilon_0$ there
exists some $t_0>0$ such that $\forall t<t_0 \quad w^{-1}(t)\cap
\mathcal{B}\subset \cup_{\varepsilon \leq \varepsilon_0} p_\varepsilon(X)$.
\end{lema}

\begin{proof} Since $X$ is compact there exists $t_0=\min_{x\in X}\{w(\bar{B}(x,\varepsilon_0))
\}$. $\forall t<t_0$ and $\forall y\in X$ there exists some
$\delta>0$ such that $\pi_t(y)=\bar{B}(y,\delta)$ and, since the
Whitney map is increasing on the trajectories,
$\delta<\varepsilon_0$.
\end{proof}

\begin{lema} For every $0<t_0$ there
exists some $\varepsilon_0>0$ such that $\forall \, \varepsilon<\varepsilon_0$
and $\forall \, x\in X \ w(\bar{B}(x,\varepsilon))<t_0$.
\end{lema}

\begin{proof} Since $X$ is compact let $\varepsilon_0=min_{x\in X}\{\varepsilon \ | \
w(\bar{B}(x,\varepsilon))=t_0\}$.
\end{proof}

We may also give a Lyapunov function from the Whitney function, obtaining also that it takes value 1 on the single points:

\begin{prop}\label{lyapunov} For any compact length space $(X,d)$ there exists a
Lyapunov function $\Phi:C(X) \to [0,1]$ with $\Phi(x)=1 \ \forall
x$ and $\Phi(X)=0$.
\end{prop}

\begin{proof} For any compact length space there exists $w:C(X)\to [0,1]$
a Whitney map such that $w(x)=0 \ \forall x \in X$ and
$\Phi(X)=1$. It suffices to define $\Phi(A):=1-w(A)$ to obtain such
a Lyapunov function. Notice that we only need the function to be
decreasing along trajectories and this comes from
\ref{increasing}.
\end{proof}

As we introduced above, now we are going to characterize the property of being topologically robust in terms of the geometry in the boundary of the $\mathbb{R}$--tree induced by the semiflow.

\begin{lema}\cite{Chi}\label{0-hip} A metric space $(X,d)$ is an $\mathbb{R}$-tree if and only if it is connected and 0-hyperbolic.
\end{lema}



\begin{definicion}\cite{Gr}\label{Gr1} Given a base point $x$ in a metric space $(X,d)$, the
\emph{Gromov product} of two points $y,z\in X$ is
\[(y\cdot z)_x=\frac{1}{2}\{d(x,y)+d(x,z)-d(y,z)\}.\]
\end{definicion}

Let us define a subset $A\subset X \times [0,diam(X))$ where the
pair $(x,t)\in A$ if $0\leq  t < t_x:=inf\{t:\bar{B}(x,t)=X\}$.
Define an equivalence relation on $A$ by $(x,t)\sim (y,t')$ if
$\bar{B}(x,t_x-t)=\bar{B}(y,t_y-t')$. Note that if $(x,t)\sim
(y,t')$, $\bar{B}(x,t_x-t)=\bar{B}(y,t_y-t')$ implies
that $d_H(X,\bar{B}(x,t_x-t))=t_x-(t_x-t)=
t_y-(t_y-t')=d_H(X,\bar{B}(y,t_y-t'))$, and therefore, $t=t'$.

Let $S=A/\sim$ and let us endow $S$ with the following metric.
$D([x,t],[y,t'])=t_x-t+t_y-t'-2min\{t_x-t,t_y-t',l(x,y)\}$ where
$l(x,y)=t_x-inf\{s:\bar{B}(x,s)\cap \alpha_y \neq \emptyset\}=
t_y-inf\{s:\bar{B}(y,s)\cap \alpha_x \neq \emptyset\}$.

It can be seen with the same method used in Proposition \ref{Prop: metric} that the metric is well defined and $(S,D)$ is an
$\mathbb{R}$--tree. Then, fixing the class $v=[(x,0)]$ which corresponds to the ball $\bar{B}(x,t_x)=X$ for every $x\in X$, $(S,v)$ is a rooted tree. By a \emph{brach} we mean any rooted non-extendable isometric embedding $f\colon [0,t)\to S$ (let us recall that rooted means that $f(0)=v$). 
Clearly, there exist a bijection between $X$ and the branches of $(S,v)$. In fact, any branch $[x\times
[0,t_x)]$ of $(S,v)$ is isometric to the restriction of the hyperspace to the segment $(\alpha_x,d_H|_{\alpha_x})$ which Kelley called \emph{segment from $\{x\}$ to $\{X\}$}, \cite{Ke}, or
\emph{order arc} according to \cite{I-N} or \cite{Ly}. Nevertheless, this tree is bounded and it is not geodesically complete. Hence, it is not suitable to represent the ramification process from a geometric point of view. Let us define it in such a way that the trajectories generate infinite branches.

Now, to define a geodesically complete $\mathbb{R}$-tree let us
consider a Lyapunov function $\Phi:C(X) \to [0,1]$ with $\Phi(x)=1
\ \forall x$ and $\Phi(X)=0$.

Let us parametrize  $\alpha_x$ as follows: $\Phi_x^{-1}:[0,1]\to
X$ where $\Phi_x^{-1}(t)=\bar{B}(x,\varepsilon(x,t))$ such that
$\Phi(\bar{B}(x,\varepsilon(x,t)))=1-t$. Note that this
$\varepsilon(x,t)$ is uniquely determined by $x$ and $t$.

Define an equivalence relation on $X\times [0,\infty)$ where
$(x,t)$ represents $\bar{B}(x,\varepsilon(x,e^{-t}))$ and $(x,t)\sim
(y,t')$ if $\varepsilon(x,e^{-t})=\varepsilon(y,e^{-t'})$ and
$\bar{B}(x,\varepsilon(x,e^{-t}))=\bar{B}(y,\varepsilon(y,e^{-t'}))$.
Note that when the balls coincide, the Lyapunov function on them
will be
$1-e^{-t}=\Phi(\bar{B}(x,\varepsilon(x,e^{-t})))=\Phi(\bar{B}(y,\varepsilon(y,e^{-t'})))
=1-e^{-t'}$ and hence, $t=t'$. Also note that when two balls with
different radius coincide they are not identified in the tree, and
the branches corresponding to their centers are disjoint from the
root on.

\begin{lema} If $(x,t)\sim (y,t)$ then $(x,t')\sim (y,t') \ \forall
t'<t$.
\end{lema}

\begin{proof} For any $t'<t$, $e^{-t'}>e^{-t}$ and let
$\bar{B}(x,\varepsilon_1')=\Phi_x^{-1}(e^{-t'})$ and
$\bar{B}(y,\varepsilon_2')=\Phi_y^{-1}(e^{-t'})$. By the properties
of the length metric, since $\bar{B}(x,\varepsilon(x,e^{-t})) =
\bar{B}(y,\varepsilon(y,e^{-t}))$, and
$\varepsilon(x,e^{-t})=\varepsilon(y,e^{-t})=\varepsilon_0$, then $\forall
\varepsilon'> \varepsilon_0$, $\bar{B}(x,\varepsilon') =
\bar{B}(y,\varepsilon')$. Therefore, one of the balls
$\bar{B}(x,\varepsilon_1'),\bar{B}(y,\varepsilon_2')$ must be contained
in the other and both are in the common part of the trajectories
$\alpha_x,\alpha_y$; but since the Lyapunov function on them takes
the same value, $e^{-t'}$, those balls must coincide and with
$\varepsilon'_i>\varepsilon_0$ and by lemma \ref{increasing}, this can
only occur if $\varepsilon'_1=\varepsilon'_2$ and thus, we finally
obtain that $(x,t')\sim (y,t')$.
\end{proof}

Let $T=X\times [0,\infty)/\sim$ and let us endow $T$ with the
following metric. \[D([x,t],[y,t'])=t+t'-2min\{t,t',m(x,y)\}
\mbox{ where } m(x,y)=sup\{s:(x,s)\sim (y,s)\}.\]

\begin{prop}\label{Prop: metric} $D$ is a metric.
\end{prop}

\begin{proof} $D$ is Well defined. Suppose $[x,t]=[x',t]$,
then we only need to show that $d([x,t],[y,t'])=d([x',t],[y,t'])$
for any $[y,t']\in T$. We can distinguish two cases.
\begin{itemize} \item[Case 1.] $sup\{s|(x,s)\sim (x',s)\}
\geq sup\{s|(x,s)\sim (y,s)\}$. Hence it is immediate to see that
$m(x,y)=m(x',y)$ and the distance is the same.

\item[Case 2.] $t\leq sup\{s|(x,s)\sim (x',s)\} < sup\{s|(x,s)\sim
(y,s)\}$. In this case, $t<m(x,y)$ and $t< m(x',y)$ and hence,
$min\{t,t',m(x,y)\}=min\{t,t'\}=min\{t,t',m(x',y)\}$ and the
distance is the same.
\end{itemize}

$D$ is a metric. \begin{itemize} \item[1)] $D\geq 0$.
It is clear that $t+t'-2 min \{t,t',m(x,y)\}\geq |t-t'| \geq 0$.

\item[2)] $D([x,t],[y,t'])=0 \Leftrightarrow [x,t]=[y,t']$. If
$D([x,t],[y,t'])=0$ then $t+t'-2 \min\{t,t',m(x,y)\}=0 \Rightarrow
\min\{t,t',m(x,y)\}=t=t'$ and since $m(x,y)\geq t=t'$ then
$[x,t]=[y,t']$.

\item[3)] Symmetric. This is obvious since the
definition is symmetric.

\item[4)] Triangle inequality,
$D([x,s],[y,t])\leq D([x,s],[z,u])+D([z,u],[y,t])$. Clearly
$t+s-2min\{s,t,m(x,y)\} \leq
s+u-2min\{s,u,m(x,z)\}+u+t-2min\{u,t,m(z,y)\} \Leftrightarrow
-min\{s,t,m(x,y)\} \leq u-min\{s,u,m(x,z)\}-min\{u,t,m(z,y)$.

Let $a=m(x,z)$, $b=m(y,z)$, $c=m(x,y)$. Clearly $min\{a,b\}\leq
c$. Without loss of generality assume that $a\leq b$ and hence $a
= min\{a,b,c\}$.

Thus we need to show  that $min\{s,u,a\}+min\{t,u,b\}\leq
min\{s,t,c\}+u$. There are three cases to consider: \item[(a)]
$u=min\{s,u,a\}$. Then $u+min\{t,u\}\leq min\{s,t,c\}+u$ because
$u\leq a \leq c$. \item[(b)] $s=min\{s,u,a\}$. Then it suffices to check that $s+min\{t,u,b\}\leq min\{s,t\}+u$. This is readily seen considering the
cases $t\leq s$ and $s\leq t$. \item[(c)] $a=min\{s,u,a\}$. Then
it is clear that $a+min\{t,u,b\}\leq min\{s,t,c\}+u$ considering
the cases $t\leq a$ and $a\leq t$.
\end{itemize}
\end{proof}

\begin{prop} $(T,D)$ is a geodesically complete $\mathbb{R}$-tree.
\end{prop}

\begin{proof} By \ref{0-hip}, it suffices to show that $T$ is connected and
0-hyperbolic in the sense of Gromov. For every point
there is a path connecting it to the root so the first part is
obvious.

The Gromov product, \ref{Gr1}, of $[x,t]$ and $[y,s]$ with respect to the
root, $w=[x,0]$ for any $x$, is given by
\[([x,t]\cdot[y,s])_w=\frac{1}{2}\{D([x,t],w)+D([y,s],w)-D([x,t],[y,s])\}\]
 Since $D([x,t],w)=t$, this means that
 $([x,t]\cdot[y,s])_w=min\{t,s,m(x,y)\}$. Given $[z,u]\in T$, this
 must be compared with
  $min\{([x,t]\cdot[z,u])_w,([z,u]\cdot[y,s])_w\}=min\{min\{t,u,m(x,z)\},
 min\{u,s,m(z,y)\}\}=min\{t,u,s,m(x,z),m(z,y)\}$.

 Thus, it suffices to check that $m(x,y)\geq min\{m(x,z),m(z,y)\}$
 which is obvious.

 Finally, to see that $(T,w)$ is geodesically complete let $\alpha:[0,t_0]\to
 (T,w)$ be an isometric embedding such that $\alpha(0)=w$. Then,
 $\alpha(t_0)=[x,t_0]$ for some $x\in X$ and by the uniqueness of
 arcs, $\alpha(t)=[x,t]$ for $0\leq t\leq t_0$ and $\alpha(t)=[x,t] \forall
 t\geq 0$ gives the desired extension of $\alpha$ to a geodesic ray.
\end{proof}

The end space of this tree endows $X$ with an ultrametric, $d_U$, where
two points are near if there is some small radius $\varepsilon$ such
that the balls centered at both points coincide. This ultrametric endows $X$ with a topology which is thinner than the initial, resulting a space which is not compact nor separable.

\begin{prop} The identity map $id: (X,d_u) \to (X,d)$ is continuous.
\end{prop}

\begin{proof} Consider any convergent sequence in $(X,d_u)$, $(x_n)\to
x$. There must be some sequence of positive real numbers
$(\varepsilon_n)\to 0$ such that
$\bar{B}(x_n,\varepsilon_n)=\bar{B}(x,\varepsilon_n)$ and, in
particular, $d(x_n,x)\leq \varepsilon_n$ converges to 0 and $(x_n)\to
x$ in $(X,d)$.
\end{proof}

\begin{prop} If $(X,d)$ is a compact length space then $(X,d_u)$
is not separable.
\end{prop}

\begin{proof} Let $x,y\in X$ be any two points and $[x,y]$ a geodesic
segment joining them. Let us see that $C=[x,y]\subset (X,d_u)$ is a
closed subset and $(C,d_u|_{C})$ is a discrete subset with
cardinal $>\aleph_0$. It is closed since $C$ is compact in
$(X,d)$ and $id:(X,d_u)\to (X,d)$ is continuous. Also, for any
$\delta<\frac{d(x,y)}{2}$ and any $z,z'\in [x,y]$, $\bar{B}(z,\delta)\neq \bar{B}(z',\delta)$ and therefore the subspace $(C,d_u|_{C})$ is uniformly discrete.
\end{proof}

The condition of being topologically robust can be characterized as follows:

\begin{prop} $(X,d_u)$ is uniformly discrete if and only $(X,d)$ is topologically robust.
\end{prop}

\begin{proof} Suppose $\delta>0$ such that $\forall x\in X$,
$B_{d_u}(x,\delta)=\{x\}$. Then let $\varepsilon_0=min_{x\in
X}\{\varepsilon_x \, | \, \Phi(\bar{B}(x,\varepsilon_x))=1-ln(\delta)\}$ which
is reached since $X$ is compact and obviously greater than 0 since
the Lyapunov function value of the ball is $1-\delta<1$.

Conversely, if the projection is injective $\forall
\varepsilon<\varepsilon_0$, let $\delta=max_{x\in
X}\{\Phi(\bar{B}(x,\varepsilon_0))\}$. Then $\delta<1$ and it is well defined
since $X$ is compact. For any point $x\in X$, the class $[x,t]$ in
the tree is not identified with any other class $[y,t]$ for any
$t>-ln(1-\delta)$ and thus, $B_{d_u}(x,1-\delta)=\{x\} \ \forall x
\in X$.
\end{proof}

This, together with \ref{flujo homeom.geod}, \ref{riemann},
\ref{curvat} respectively, implies the following corolaries.

\begin{cor} If $(X,d)$ is an r-perfectly geodesic compact length space
then $(X,d_u)$ is uniformly discrete.
\end{cor}

\begin{cor} If $(X,d)$ is a compact connected Riemannian manifold with the natural length metric 
then $(X,d_u)$ is uniformly discrete.
\end{cor}

\begin{cor} If $(X,d)$ is a uniquely geodesic compact length space with curvature bounded below
then $(X,d_u)$ is uniformly discrete.
\end{cor}

\end{document}